\documentclass[11pt]{article}
\usepackage{latexsym,amsfonts,amsmath,amsthm,amssymb, epsfig}

\usepackage[vcentering,dvips]{geometry}
\newtheorem{thm}{Theorem}
\newtheorem{lem}[thm]{Lemma}

\newtheorem{cor}{Corollary}

\theoremstyle{remark}
\newtheorem{rem}{Remark}

\theoremstyle{definition}
\newtheorem{dfn}[thm]{Definition}

\newcommand{\vek}[1]{\boldsymbol{#1}}
\newcommand{\norm}[2]{\left\|\left.{#1}\right|{#2}\right\|}

\newcommand{\R}{\mathbb{R}}
\newcommand{\N}{\mathbb{N}}
\newcommand{\Z}{\mathbb{Z}}
\newcommand{\Sn}{\mathcal{S}(\R^n)}
\newcommand{\SSn}{\mathcal{S}'(\R^n)}
\newcommand{\supp}{{\rm supp\ }}

\newcommand{\ellqp}{{\ell_{q(\cdot)}(L_{p(\cdot)})}}
\newcommand{\ellpq}{{L_{p(\cdot)}(\ell_{q(\cdot)})}}
\newcommand{\p}{{p(\cdot)}}
\newcommand{\q}{{q(\cdot)}}
\newcommand{\s}{{s(\cdot)}}

\newcommand{\mgk}{\mathcal{W}^\alpha_{\alpha_1,\alpha_2}}
\newcommand{\Lpp}{L_\p(\R^n)}
\newcommand{\Bwpqpunkt}{B^{\vek{w}}_{\p,\q}(\R^n)}
\newcommand{\Fwpqpunkt}{F^{\vek{w}}_{\p,\q}(\R^n)}
\newcommand{\Bwpq}{B^{\vek{w}}_{p,q}(\R^n)}
\newcommand{\Fwpq}{F^{\vek{w}}_{p,q}(\R^n)}
\newcommand{\Bspqpunkt}{B^{s(\cdot)}_{\p,\q}(\R^n)}
\newcommand{\Fspqpunkt}{F^{s(\cdot)}_{\p,\q}(\R^n)}
\newcommand{\FT}{\mathcal{F}}
\newcommand{\IFT}{\mathcal{F}^{-1}}
\newcommand{\tah}{^\vee}
\renewcommand{\P}{\mathcal{P}(\R^n)}
\newcommand{\Plog}{\mathcal{P}^{\log}(\R^n)}

\newcommand{\esssup}{\operatornamewithlimits{ess-sup}}
\newcommand{\essinf}{\operatornamewithlimits{ess-inf}}

\title{Spaces of variable smoothness and integrability: Characterizations by local means and ball means of differences}

\author{Henning Kempka\footnote{Mathematical Institute, Friedrich-Schiller-University Jena, D--07737 Jena, Germany,
email: {\tt henning.kempka@uni-jena.de}},
Jan Vyb\'\i ral\footnote{Johann Radon Institute for Computational and
Applied Mathematics, Austrian Academy of Sciences, Altenbergerstrasse 69, A--4040 Linz, Austria,
email: {\tt  jan.vybiral@oeaw.ac.at}.}}

\begin{document}

\maketitle

\begin{abstract}
We study the spaces $B^{s(\cdot)}_{p(\cdot),q(\cdot)}(\R^n)$ and $F^{s(\cdot)}_{p(\cdot),q(\cdot)}(\R^n)$ of Besov and Triebel-Lizorkin 
type as introduced recently in \cite{AlmeidaHasto} and \cite{DHR}. Both scales cover many classical spaces with fixed exponents as well 
as function spaces of variable smoothness and function spaces of variable integrability.\\
The spaces $B^{s(\cdot)}_{p(\cdot),q(\cdot)}(\R^n)$ and $F^{s(\cdot)}_{p(\cdot),q(\cdot)}(\R^n)$  have been introduced in 
\cite{AlmeidaHasto} and \cite{DHR} by Fourier analytical tools, as the decomposition of unity. Surprisingly, our main result 
states that these spaces also allow a characterization in the time-domain with the help of classical ball means of differences.

To that end, we first prove a local means characterization for $B^{s(\cdot)}_{p(\cdot),q(\cdot)}(\R^n)$ with the help of the so-called
Peetre maximal functions. Our results do also hold for 2-microlocal function spaces $\Bwpqpunkt$ and $\Fwpqpunkt$ which are a slight
generalization of generalized smoothness spaces and spaces of variable smoothness.

\end{abstract}



\section{Introduction}

Function spaces of variable integrability appeared in a work by Orlicz \cite{Orlicz} already in 1931, but the recent interest in these
spaces is based on the paper of Kov\'a\v{c}ik and R\'akosnik \cite{KoRa} together with applications
in terms of modelling electrorheological fluids \cite{Ruz1}. A fundamental breakthrough concerning spaces of variable integrability was the observation that, under certain regularity assumptions on $\p$, the Hardy-Littlewood maximal operator is also bounded on $L_{\p}(\R^n)$, see \cite{Diening04}. This result has been generalized to wider classes of exponents $\p$ in \cite{CruzUribe03}, \cite{Nekvinda04} and \cite{DieningHHMS09}.\\
Besides electrorheological fluids, the spaces $L_{\p}(\R^n)$ possess interesting applications in the theory of PDE's, variational calculus, financial mathematics and image processing. A recent overview of this vastly growing field is given in \cite{DHHR}.\\
Sobolev and Besov spaces with variable smoothness but fixed integrability have been introduced in the late 60's and early 70's in the works of Unterberger \cite{Unterberger}, Vi\v{s}ik and Eskin \cite{VisikEskin}, Unterberger and Bokobza \cite{UnterbergerBok} and in the work of Beauzamy \cite{Beauzamy}. Leopold studied in \cite{Leopold91} Besov spaces where the smoothness is determined by a symbol $a(x,\xi)$ of a certain class of hypoelliptic pseudodifferential operators. In the special case $a(x,\xi)=(1+|\xi|^2)^{\sigma(x)/2}$ these spaces coincide with spaces of variable smoothness $B^{\sigma(x)}_{p,p}(\R^n)$.\\
A more general approach to spaces of variable smoothness are the so-called 2-microlocal function spaces $\Bwpq$ and $\Fwpq$. Here the smoothness in these spaces gets measured by a weight sequence $\vek{w}=(w_j)_{j=0}^\infty$. Besov spaces with such weight sequences appeared first in the works of Peetre \cite{PeetreArt} and Bony \cite{Bony}. Establishing a wavelet characterization for 2-microlocal H\"older-Zygmund spaces in \cite{Jaffard91} it turned out that 2-microlocal spaces are well adapted in connection to regularity properties of functions (\cite{JaffardMeyer96},\cite{Meyer97},\cite{VehelSeuret04}). Spaces of variable smoothness are a special case of 2-microlocal function spaces and in \cite{VehelSeuret03} and \cite{Bes1} characterizations by differences have been given for certain classes of them.\\\\
The theories of function spaces with fixed smoothness and variable integrability 
and function spaces with variable smoothness and fixed integrability finally crossed
each other in \cite{DHR}, where the authors introduced the function spaces of Triebel-Lizorkin type
with variable smoothness and simultaneously with variable integrability. 
It turned out that many of the spaces mentioned above
are really included in this new structure, see \cite{DHR} and references therein.
The key point to merge both lines of investigation was the study of traces. 
From Theorem 3.13 in \cite{DHR} 
\begin{align*}
    tr_{\R^{n-1}}F^{\s}_{\p,\q}(\R^n)=F^{\s-1/\p}_{\p,\p}(\R^{n-1})
\end{align*}
one immediately understands the necessity to take all exponents variable assuming $\p$ or $\s$ variable. So the trace embeddings may be described in a natural way in the context of these spaces.
Furthermore, this was complemented in \cite{Vyb1} by showing,
that the classical Sobolev embedding theorem
\begin{align*}
	F^{s_0(\cdot)}_{p_0(\cdot),q(\cdot)}(\R^n)\hookrightarrow F^{s_1(\cdot)}_{p_1(\cdot),q(\cdot)}(\R^n)
\end{align*}
 holds also in this scale of function spaces
if the usual condition is replaced by its point-wise analogue
\begin{align*}
	s_0(x)-n/p_0(x)=s_1(x)-n/p_1(x),\quad x\in\R^n.
\end{align*}
Finally, Almeida and H\"ast\"o managed in \cite{AlmeidaHasto} to adapt the definition of Besov spaces to
the setting of variable smoothness and integrability and proved the Sobolev and other usual embeddings in this scale.

The properties of Besov and Triebel-Lizorkin spaces of variable smoothness and integrability
known so far give a reasonable hope that these new scales of function spaces enjoy sufficiently
many properties to allow a local description of many effects, which up to now could only be described in a global way. 
Subsequently, for the spaces $\Fspqpunkt$ there is a characterization by local means given in \cite{Kempka09}. This characterization still works with Fourier analytical tools but the analyzing functions $k_0,k\in\Sn$ are compactly supported in the time-domain and we only need local values of $f$ around $x\in\R^n$ to calculate the building blocks $k(2^{-j},f)(x)$. This is in sharp contrast to the definition of the spaces by the decomposition of unity, cf. Definitions \ref{dfn_decomp} and \ref{def:BFpunkt}. For the spaces $\Bspqpunkt$ we will prove a local means assertion of this type in Section \ref{sec:localmeans} which will be helpful later on.\\
The main aim of this paper is to present another essential property of the function spaces from \cite{DHR}
and \cite{AlmeidaHasto}. We prove the surprising result that these spaces $\Bspqpunkt$ and $\Fspqpunkt$ with variable smoothness and integrability do also allow a characterization purely in the time-domain by classical ball means of differences.\\
The paper is organized as follows. First of all we provide all necessary notation in Section \ref{sec:notation}. Since the proofs for spaces of variable smoothness and 2-microlocal function spaces work very similar (see Remark \ref{rem:2mlsx}) we present our results for both scales. The proof 
for the local means characterization will be given in Section \ref{sec:localmeans} in terms of 2-microlocal function spaces and we present the version for spaces of variable smoothness in Section \ref{SecLocsx}. In Section 4 we prove the characterization by ball means of differences for $\Bspqpunkt$ and $\Fspqpunkt$ and the version for 2-microlocal function spaces will be given in Section \ref{sec:2mldiff}.

\section{Notation}\label{sec:notation}
In this section we collect all the necessary definitions.
We start with the variable Lebesgue spaces $\Lpp$.
A measurable function $p:\R^n\to(0,\infty]$ is called a
\emph{variable exponent function} if it is bounded away from zero, i.e. if $p^-=\essinf_{x\in \R^n}p(x)>0$.
We denote the set of all variable exponent functions by $\P$. We put also $p^+=\esssup_{x\in \R^n}p(x)$.

The variable exponent Lebesgue space $L_{p(\cdot)}(\R^n)$ consists of all measurable functions $f$ for which there exist $\lambda>0$ such that the modular
\begin{align*}
\varrho_{L_{p(\cdot)}(\R^n)}(f/\lambda)=\int_{\R^n}\varphi_{p(x)}\left(\frac{|f(x)|}{\lambda}\right)dx
\end{align*}
is finite, where
$$
\varphi_p(t)=\begin{cases}t^p&\text{if}\ p\in(0,\infty),\\
0&\text{if}\ p=\infty\ \text{and}\ t\le 1,\\
\infty&\text{if}\ p=\infty\ \text{and}\ t>1.
\end{cases}
$$
If we define $\R^n_\infty=\{x\in\R^n:p(x)=\infty\}$ and $\R^n_0=\R^n\setminus\R^n_\infty$, then the
Luxemburg norm of a function $f\in L_{p(\cdot)}(\R^n)$ is given by
\begin{align*}
 \norm{f}{L_{p(\cdot)}(\R^n)}&=\inf\{\lambda>0:\varrho_{L_{p(\cdot)}(\R^n)}(f/\lambda)\leq1\}\\
	&=\inf\left\{\lambda>0:\!\!\!\int_{\R^n_0}\left(\frac{f(x)}{\lambda}\right)^{p(x)}\!\!\!\!dx<1\text{ and }|f(x)|<\lambda\text{ for a.e. $x\in\R^n_\infty$}\right\}.
\end{align*}
If $p(\cdot)\geq1$, then it is a norm otherwise it is always a quasi-norm.

To define the mixed spaces $\ellqp$ we have to define another modular. For $p,q\in\P$ and a
sequence $(f_\nu)_{\nu\in\N_0}$ of $\Lpp$ functions we define
\begin{align}\label{Modular_ellqp}
\varrho_\ellqp(f_\nu)=\sum_{\nu=0}^\infty\inf\left\{\lambda_\nu>0:\varrho_\p\left(\frac{f_\nu}{\lambda_\nu^{1/q(\cdot)}}\right)\leq1\right\}.
\end{align}
If $q^+<\infty$, then we can replace \eqref{Modular_ellqp} by the simpler expression
\begin{align*}
\varrho_\ellqp(f_\nu)=\sum_{\nu}\norm{|f_\nu|^{\q}}{L_{\frac{\p}{\q}}}.
\end{align*}
The (quasi-)norm in the $\ellqp$ spaces is defined as usual by
\begin{align*}
\norm{f_\nu}{\ellqp}=\inf\{\mu>0:\varrho_\ellqp(f_\nu/\mu)\leq1\}.
\end{align*}
It is known, cf. \cite{AlmeidaHasto,KV11}, that $\ellqp$ is a norm if $\q\geq1$ is constant almost everywhere (a.e.) on $\R^n$ and $\p\geq1$,
or if $1/p(x)+1/q(x)\le 1$ a.e. on $\R^n$, or if $1\le q(x)\le p(x)\le \infty$ a.e. on $\R^n$.
Surprisingly enough, it turned out in \cite{KV11} that the condition $\min(p(x),q(x))\ge 1$ a.e. on $\R^n$ is not sufficient
for $\ellqp$ to be a norm. Nevertheless, it was proven in \cite{AlmeidaHasto} that it is a quasi-norm for every $p,q\in\P.$

For the sake of completeness, we state also the definition of the space $\ellpq$,
which is much more intuitive then the definition of $\ellqp$. One just takes the $\ell_{q(x)}$
norm of $(f_\nu(x))_{\nu\in\N_0}$ for every $x\in\R^n$ and then the $L_\p$-norm with respect to $x\in\R^n$, i.e.
$$
\norm{f_\nu}{\ellpq}=\big\|\norm{f_\nu(x)}{\ell_{q(x)}}|L_{\p}\big\|.
$$
It is easy to show (\cite{DHR}) that $\ellpq$ is always a quasi-normed space and it is a normed space, if $\min(p(x),q(x))\geq1$ holds point-wise.\\
The summation in the definition of the norms of $\ellqp$ and $\ellpq$ can also be taken for $\nu\in\Z$. 
It always comes out of the context over which interval the summation is taken.
Occasionally, we may indicate it by $\norm{(f_\nu)_{\nu=-\infty}^\infty}{\ellqp}$.\\
By $\hat{f}=\FT f$ and $f\tah=\IFT f$ we denote the usual Fourier transform and its inverse on $\Sn$, the Schwartz space of
smooth and rapidly decreasing functions, and on $\SSn$, the dual of the Schwartz space.

\subsection{Spaces $\Bspqpunkt$ and $\Fspqpunkt$}

The definition of Besov and Triebel-Lizorkin spaces of variable smoothness and integrability is based
on the technique of \emph{decomposition of unity} exactly in the same manner as in the case of constant
exponents.

\begin{dfn}\label{dfn_decomp}
Let $\varphi_0\in\Sn$ with $\varphi_0(x)=1$ for $|x|\leq1$ and
$\supp\varphi_0\subseteq\{x\in\R^n:|x|\leq2\}$. For $j\geq1$ we define
\begin{align*}
\varphi_j(x)=\varphi_0(2^{-j}x)-\varphi_0(2^{-j+1}x).
\end{align*}
\end{dfn}
One may verify easily that
$$
\sum_{j=0}^\infty \varphi_j(x)=1\quad\text{for all}\quad x\in\R^n.
$$
The following regularity classes for the exponents are necessary to make the definition of
the spaces independent on the chosen decomposition of unity.
\begin{dfn}
Let $g\in C(\R^n)$.

\begin{itemize}
\item[(i)] We say that $g$ is \emph{locally $\log$-H\"older continuous}, abbreviated $g\in C^{\log}_{loc}(\R^n)$,
if there exists $c_{\log}(g)>0$ such that
\begin{align}\label{eq:defclog}
|g(x)-g(y)|\leq\frac{c_{\log}(g)}{\log(e+{1}/{|x-y|})}
\end{align}
holds for all $x,y\in\R^n$.
\item[(ii)]
We say that $g$ is \emph{globally $\log$-H\"older continuous}, abbreviated $g\in C^{\log}(\R^n)$, if $g$ is
locally $\log$-H\"older continuous and there exists $g_\infty\in\R$ such that
\begin{align*}
|g(x)-g_\infty|\leq\frac{c_{\log}}{\log(e+|x|)}
\end{align*}
holds for all $x\in\R^n$.
\end{itemize}
\end{dfn}
\begin{rem}
With \eqref{eq:defclog} we obtain
\begin{align*}
|g(x)|\leq c_{\log}(g)+|g(0)|,\ \text{for all $x\in\R^n$.}
\end{align*}
This implies that all functions $g\in C^{\log}_{loc}(\R^n)$ always belong to $L_\infty(\R^n)$
\end{rem}
If an exponent $p\in\P$ satisfies $1/p\in C^{\log}(\R^n)$, then we say it belongs to the class $\Plog$.
We recall the definition of the spaces $\Bspqpunkt$ and $\Fspqpunkt$, as given in \cite{DHR} and \cite{AlmeidaHasto}.
\begin{dfn}\label{def:BFpunkt}
(i) Let $p,q\in \Plog$ with $0<p^-\leq p^+<\infty$, $0<q^-\leq q^+<\infty$ and let $s\in C^{\log}_{loc}(\R^n)$. Then
\begin{align*}
\Fspqpunkt&=\left\{f\in
\SSn:\norm{f}{\Fspqpunkt}_\varphi<\infty\right\}\text{, where}\\
\norm{f}{\Fspqpunkt}_\varphi&=\norm{2^{js(\cdot)}(\varphi_j\hat{f})\tah}{\ellpq}.
\end{align*}
(ii) Let $p,q\in \Plog$ and let $s\in C^{\log}_{loc}(\R^n)$. Then
\begin{align*}
\Bspqpunkt&=\left\{f\in
\SSn:\norm{f}{\Bspqpunkt}_\varphi<\infty\right\}\text{, where}\\
\norm{f}{\Bspqpunkt}_\varphi&=\norm{2^{js(\cdot)}(\varphi_j\hat{f})\tah}{\ellqp}.
\end{align*}
\end{dfn}
The subscript $\varphi$ at the norm symbolizes that the definition formally does depend on the resolution of unity.
From \cite{Kempka09} and \cite{AlmeidaHasto} we have that the definition of the spaces $\Fspqpunkt$ and
$\Bspqpunkt$ is independent of the chosen
resolution of unity if $p,q\in\Plog$ and $s\in C^{\log}_{loc}(\R^n)$. That means that different start functions $\varphi_0$ and $\tilde{\varphi}_0$ from Definition \ref{dfn_decomp} induce equivalent norms in the above definition. So we will suppress the subscript $\varphi$ in the notation of the norms.\\
Let us comment on the conditions on $p,q\in\Plog$ for the Triebel-Lizorkin spaces. The condition $0<p^-\leq p^+<\infty$ is quite natural since there exists also the restriction $p<\infty$ in the case of constant exponents, see \cite{Triebel1} and \cite{SYY}. The second one, $0<q^-\leq q^+<\infty$, is a bit unnatural and comes from the use of the convolution Lemma \ref{lem:ConvDHR} (\cite[Theorem 3.2]{DHR}). There is some hope that this convolution lemma can be generalized and the case $q^+=\infty$ can be incorporated in the definition of the $F$-spaces.\\

The Triebel-Lizorkin spaces with variable smoothness have first been introduced in \cite{DHR} under much more restrictive conditions on $\s$. These conditions have been relaxed in \cite{Kempka09} in the context of 2-microlocal function spaces (see the next subsection).\\
Besov spaces with variable $\p,\q$ and $\s$ have been introduced in \cite{AlmeidaHasto}.\\
Both scales contain as special cases a lot of well known function spaces. If $s,p$ and $q$ are constants, then we derive the well known Besov and Triebel-Lizorkin spaces with usual H\"older and Sobolev spaces included, see \cite{Triebel1} and \cite{Triebel2}. If the smoothness $s\in\R$ is a constant and $p\in\Plog$ with $p^->1$, then $F^s_{\p,2}(\R^n)=\mathcal{L}^s_\p(\R^n)$ are the variable Bessel potential spaces from \cite{AlmeidaSamko06} and \cite{GurkaHarjNek07} with its special cases $F^0_{\p,2}(\R^n)=L_\p(\R^n)$ and $F^k_{\p,2}(\R^n)=W^k_\p(\R^n)$ for $k\in\N_0$, see \cite{DHR}.\\
Taking $s\in\R$ and $q\in(0,\infty]$ as constants we derive the spaces $F^s_{\p,q}(\R^n)$ and $B^s_{\p,q}(\R^n)$ studied by Xu in \cite{Xu1} and \cite{Xu2}.\\
Furthermore it holds $F^\s_{\p,\p}(\R^n)=B^\s_{\p,\p}(\R^n)$ and $B^\s_{\infty,\infty}(\R^n)$ equals the variable H\"older-Zygmund space $\mathcal{C}^\s(\R^n)$ introduced in \cite{AlmeidaSamko07}, \cite{AlmeidaSamko09} and \cite{RossSamko95} with $0<s^-\leq s^+\leq1$, see \cite{AlmeidaHasto}.

\subsection{2-microlocal spaces}\label{sec:not2ml}
The definition of Besov and Triebel-Lizorkin spaces of variable smoothness and integrability
is a special case of the so-called 2-microlocal spaces of variable integrability.
As some of the results presented here get proved in this more general scale,
we present also the definition of 2-microlocal spaces. It is based on the dyadic decomposition of unity as presented
above combined with the concept of admissible weight sequences.

\begin{dfn}\label{dfn_zulGewichte}
Let $\alpha\geq0$ and let $\alpha_1,\alpha_2\in\R$ with $\alpha_1\leq\alpha_2$. A sequence of non-negative measurable
functions $\vek{w}=(w_j)_{j=0}^\infty$ belongs to the class
$\mgk$ if and only if
\begin{enumerate}
\item[(i)] there exists a constant $C>0$ such that
\begin{align*}
0<w_j(x)\leq C w_j(y)\left(1+2^{j}|x-y|\right)^{\alpha}\quad\text{for all $j\in\N_0$ and all $x,y\in\R^n$}
\end{align*}
\item[(ii)] and for all $j\in\N_0$ and all $x\in\R^n$ we have
\begin{align*}
2^{\alpha_1}w_j(x)\leq w_{j+1}(x)\leq 2^{\alpha_2}w_j(x).
\end{align*}
\end{enumerate}
Such a system $(w_j)_{j=0}^\infty\in\mgk$ is called an \emph{admissible weight sequence}.
\end{dfn}
Finally, here is the definition of the spaces under consideration.
\begin{dfn}\label{def:W_Spaces}
Let $\vek{w}=(w_j)_{j\in\N_0}\in\mgk$. Further, let $p,q\in\Plog$ (with $p^+,q^+<\infty$ in the $F$-case), then we define
\begin{align*}
\Bwpqpunkt&=\left\{f\in
\SSn:\norm{f}{\Bwpqpunkt}_\varphi<\infty\right\}\text{, where}\\
\norm{f}{\Bwpqpunkt}_\varphi&=\norm{w_j(\varphi_j\hat{f})\tah}{\ellqp}
\end{align*}
and
\begin{align*}
\Fwpqpunkt&=\left\{f\in
\SSn:\norm{f}{\Fwpqpunkt}_\varphi<\infty\right\}\text{, where}\\
\norm{f}{\Fwpqpunkt}_\varphi&=\norm{w_j(\varphi_j\hat{f})\tah}{\ellpq}.
\end{align*}
\end{dfn}
The independence of the decomposition of unity for the 2-microlocal spaces from Definition \ref{def:W_Spaces} follows from the local means characterization (see \cite{Kempka09} for the Triebel-Lizorkin and Section \ref{sec:localmeans} for the Besov spaces).\\
The 2-microlocal spaces with the special weight sequence
\begin{align}\label{tmlweights}
w_j(x)=2^{js}(1+2^j|x-x_0|)^{s'}\qquad\text{ with }s,s'\in\R\text{ and }x_0\in\R^n
\end{align}
have first been introduced by Peetre in \cite{PeetreArt} and by Bony in \cite{Bony}. Later on, Jaffard and Meyer gave a characterization in \cite{Jaffard91} and  \cite{JaffardMeyer96} with wavelets of the spaces $C^{s,s'}_{x_0}=B^{\vek{w}}_{\infty,\infty}(\R^n)$ and $H^{s,s'}_{x_0}=B^{\vek{w}}_{2,2}(\R^n)$ with the weight sequence \eqref{tmlweights}. It turned out that spaces of this type are very useful to study regularity properties of functions. Subsequently, L\'evy-V\'ehel and Seuret developed in \cite{VehelSeuret04} the 2-microlocal formalism and studied the behavior of cusps, chirps and fractal functions with respect to the spaces $C^{s,s'}_{x_0}$.\\
A first step to a more general weight sequence $\vek{w}$ has been taken by Moritoh and Yamada in \cite{Moritoh} and wider ranges of function spaces have been studied by Xu in \cite{HongXu} and by Andersson in \cite{Andersson}.\\
The above definition for 2-microlocal weight sequences was presented by Besov in \cite{Bes1} and also in \cite{Kempka09} by Kempka.\\
A different line of study for spaces of variable smoothness - using different methods - are the spaces of generalized smoothness introduced by Goldman and Kalyabin in \cite{Go79}, \cite{Go80}, \cite{Ka77a} and \cite{Ka80}. A systematic treatment of these spaces based on differences has been given by Goldman in \cite{Go84a}, see also the survey \cite{KaLi87} and references therein.\\
Later on, spaces of generalized smoothness appeared in interpolation theory and have been investigated in \cite{Me83}, \cite{CoFe86} and \cite{Mo01}. For further information on these spaces see the survey paper \cite{FarLeo} where also a characterization by atoms and local means for these spaces is given.\\ 
From the definition of admissible weight sequences, $d_1\sigma_j\leq\sigma_{j+1}\leq d_2\sigma_j$, it follows directly that the spaces of generalized smoothness $B^{(\sigma_j)}_{p,q}(\R^n)$ and $F^{(\sigma_j)}_{p,q}(\R^n)$ of Farkas and Leopold \cite{FarLeo} and $B^{(s,\Psi)}_{p,q}(\R^n)$ and $F^{(s,\Psi)}_{p,q}(\R^n)$ from Moura \cite{Mo01} are a special subclass of 2-microlocal function spaces with $2^{\alpha_1}=d_1$, $2^{\alpha_2}=d_2$ and $\alpha=0$.\\  
In a different approach Schneider in \cite{Schneider07} studied spaces of varying smoothness. Here the smoothness at a point gets determined by a global smoothness $s_0\in\R$ and a local smoothness function $\s$. These spaces can not be incorporated into the scale of 2-microlocal function spaces, but there exist some embeddings.
\begin{rem}\label{rem:2mlsx}
	Surprisingly, these 2-microlocal weight sequences are directly connected to variable smoothness functions $s:\R^n\to\R$ if we set
\begin{align}\label{tmlsx}
w_j(x)=2^{js(x)}.
\end{align}
If $s\in C^{\log}_{loc}(\R^n)$ (which is the standard condition on $\s$), then $\vek{w}=(w_j(x))_{j\in\N_0}=(2^{js(x)})_{j\in\N_0}$ belongs to $\mgk$ with $\alpha_1=s^-$ and $\alpha_2=s^+$. For the third index $\alpha$ we use Lemma \ref{Lem:eta1} with $m=0$ and obtain $\alpha=c_{\log}(s)$, where $c_{\log}(s)$ is the constant for $\s$ from \eqref{eq:defclog}. That means that spaces of variable smoothness from Definition \ref{def:BFpunkt} are a special case of 2-microlocal function spaces from Definition \ref{def:W_Spaces}. Both types of function spaces are very closely connected and the properties used in the proofs are either
\begin{align}\label{sxClogeigenschaft}
	2^{k|s(x)-s(y)|}\leq c\quad\text{or}\quad\frac{w_k(x)}{w_k(y)}\leq c
\end{align}
for $|x-y|\leq c2^{-k}$. This property follows directly either from the definition of $s\in C^{\log}_{loc}(\R^n)$ or from Definition \ref{dfn_zulGewichte}.\\
Nevertheless there exist examples of admissible weight sequences which can not be expressed in terms of variable smoothness functions. For example the important and well studied case of the weight sequence $\vek{w}$ from \eqref{tmlweights} can not be expressed via \eqref{tmlsx} if $s'\neq0$. Another example are the spaces of generalized smoothness which can not be identified as spaces of variable smoothness.
\end{rem}
Since spaces of variable smoothness are included in the scale of 2-microlocal function spaces all special cases of the previous subsection can be identified in the definition of 2-microlocal spaces.\\
 Although the 2-microlocal spaces include the scales of spaces of variable smoothness, we will give some of our proofs in the notation of variable smoothness, since this notation is more common. We will then reformulate the results in terms of 2-microlocal spaces, the proof works then very similar; we just have to use \eqref{sxClogeigenschaft}.

\section{Local means characterization}\label{sec:localmeans}
The main result of this section is the local means characterization of the spaces $\Bwpqpunkt$. For the spaces $\Fwpqpunkt$ there already exists a local means characterization (\cite[Corollary 4.7]{Kempka09}). We shall first give the full proof for the 2-microlocal spaces and later on (in Section \ref{SecLocsx})
we restate the result also for spaces $\Bspqpunkt$ and $\Fspqpunkt$.

The crucial tool will be the Peetre maximal operator, as defined by Peetre in \cite{PeetreArt}.
The operator assigns to each system
$(\Psi_k)_{k\in\N_0}\subset\Sn$, to each distribution $f\in\SSn$
and to each number $a>0$ the following quantities
\begin{align}\label{_PO_modifiziert}
(\Psi_k^*f)_a(x):=\sup_{y\in\R^n}\frac{|(\Psi_k\ast f)(y)|}{1+|2^k(y-x)|^{a}},\quad x\in\R^n
\text{ and }k\in\N_0.
\end{align}

We start with two given functions $\psi_0,\psi_1\in\Sn$. We define
\begin{align*}
\psi_j(x)=\psi_1(2^{-j+1}x),\quad\text{for $x\in\R^n$ and $j\in\N$.}
\end{align*}
Furthermore, for all $j\in\N_0$ we write $\Psi_j=\hat{\psi_j}$.
The main theorem of this section reads as follows.
\begin{thm}\label{Theorem:_LM_abstrakt}
Let $\vek{w}=(w_k)_{k\in\N_0}\in\mgk$, $p,q\in\Plog$ and let $a>0$, $R\in\N_0$ with $R>\alpha_2$.
Further, let $\psi_0,\psi_1$ belong to $\Sn$ with
\begin{align}
D^\beta\psi_1(0)=0,\quad\text{ for }0\leq|\beta|< R,\label{_LM_MomentCond}
\end{align}
and
\begin{align}
|\psi_0(x)|&>0\quad\text{on}\quad\{x\in\R^n:|x|<\varepsilon\}\label{_LM_Tauber1,5}\\
|\psi_1(x)|&>0\quad\text{on}\quad\{x\in\R^n:\varepsilon/2<|x|<2\varepsilon\}\label{_LM_Tauber2,5}
\end{align} 
for some $\varepsilon>0$. For $a>\frac{n+c_{log}(1/q)}{p^-}+\alpha$ and all $f\in\SSn$ we have
\begin{align*}
\norm{f}{\Bwpqpunkt}\approx\norm{(\Psi_k\ast f)w_k}{\ellqp}\approx\norm{(\Psi_k^*f)_aw_k}{\ellqp}.
\end{align*}
\end{thm}
\begin{rem}
\begin{enumerate}
\item[(i)] The proof relies on \cite{Rychkov} and will be shifted to the next section. Moreover,
Theorem \ref{Theorem:_LM_abstrakt} shows that the definition of the 2-microlocal spaces of variable integrability
is independent of the resolution of unity used in the Definition \ref{def:W_Spaces}.
\item[(ii)] The conditions \eqref{_LM_MomentCond} are usually called \emph{moment conditions} while
\eqref{_LM_Tauber1,5} and \eqref{_LM_Tauber2,5} are the so called \emph{Tauberian conditions}.
\item[(iii)] If $R=0$, then there are no moment conditions \eqref{_LM_MomentCond} on $\psi_1$.
\item[(iv)] The notation $c_{log}(1/q)$ stands for the constant from \eqref{eq:defclog} with $1/\q$. 
\end{enumerate}
\end{rem}
Next we reformulate the abstract Theorem \ref{Theorem:_LM_abstrakt}
in the sense of classical local means (see Sections 2.4.6 and 2.5.3 in \cite{Triebel2}). Since the proof is the same as the one from Theorem 2.4 in \cite{Kempka09} we just state the result. 
\begin{cor}
There exist functions $k_0,k\in\Sn$ with $\supp k_0,\supp k\subset\{x\in\R^n:|x|<1\}$ and $D^\beta\hat{k}(0)=0$ for all $0\leq|\beta|<\alpha_2$ such that for all $f\in\SSn$
\begin{align*}
\norm{k_0(1,f)w_0}{\Lpp}+\norm{k(2^{-j},f)w_j}{\ellqp}
\end{align*}
is an equivalent norm on $\Bwpqpunkt$.\\
The building blocks get calculated by
\begin{align*}
k(t,f)(x)=\int_{\R^n}k(y)f(x+ty)dy=t^{-n}\int_{\R^n}k\left(\frac{y-x}{t}\right)f(y)dy
\end{align*}
and similarly for $k_0(1,f)(x)$.
\end{cor}
A similar characterization for $\Fwpqpunkt$ and details how these functions $k_0,k\in\Sn$ can be constructed can be found in \cite{Kempka09}.

\subsection{Proof of local means}
The proof of Theorem \ref{Theorem:_LM_abstrakt} is divided into three parts. The next section
is devoted to some technical lemmas needed later. Section \ref{Sec:Comp} is devoted to the proof
of Theorem \ref{Theorem:Comp_PO_s}, which gives an inequality between different Peetre maximal operators. 
Finally, Section \ref{Sec:Bound} proves the boundedness of the Peetre maximal operator in Theorem \ref{Theorem:_Bo_PMO}.
These two theorems combined give immediately the proof of Theorem \ref{Theorem:_LM_abstrakt}.

\subsubsection{Helpful lemmas}
Before proving the local means characterization we recall some technical lemmas, which appeared in the paper of
Rychkov \cite{Rychkov}. For some of them we need adapted versions to our situation.\\
The first lemma describes the use of the so called moment conditions.

\begin{lem}[\cite{Rychkov}, Lemma 1]\label{_LM_Lemma1}
Let $g,h\in\Sn$ and let $M\in \N_0$. Suppose that
\begin{align*}
(D^\beta\hat{g})(0)=0\quad\text{for}\quad0\leq|\beta|<M.
\end{align*}
Then for each $N\in\N_0$ there is a constant $C_N$ such that
\begin{align*}
\sup_{z\in\R^n}|(g_t\ast h)(z)|(1+|z|^N)\leq C_N t^{M},\quad\text{for}\quad 0<t<1,
\end{align*}
where $g_t(x)=t^{-n}g(x/t)$.
\end{lem}

The next lemma is a discrete convolution inequality. We formulate it in a rather abstract notation and point out later on the conclusions we need.
\begin{lem}\label{lem:convolutionSeq}
Let $X\subset\{(f_k)_{k\in\Z}:f_k:\R^n\to[-\infty,\infty]\text{ measurable}\}$ be a quasi-Banach space of sequences of 
measurable functions. Further we assume that its quasi-norm is shift-invariant, i.e. it satisfies
\begin{align*}
\norm{(f_{k+l})_{k\in\Z}}{X}=\norm{(f_{k})_{k\in\Z}}{X}\qquad \text{for every $l\in\Z$ and $(f_k)_{k\in\Z}\in X$.}
\end{align*} 
For a sequence of non-negative functions $(g_k)_{k\in\Z}\in X$ and $\delta>0$ we denote 
\begin{align*}
 G_\nu(x)=\sum_{k=-\infty}^\infty2^{-|\nu-k|\delta}g_k(x),\qquad x\in\R^n,\,\nu\in\Z.
\end{align*}
Then there exists a constant $c>0$ depending only on $\delta$ and $X$ such that for every sequence $(g_k)_{k\in\Z}$
\begin{align*}
\norm{(G_\nu)_\nu}{X}\leq c\norm{(g_k)_k}{X}.
\end{align*}
\end{lem}
\begin{proof}
Since $X$ is a quasi-Banach space, there exists a $r>0$ such that $\norm{\cdot}{X}$ is equivalent to some $r$-norm, cf. \cite{Aoki, Rol}. 
We have then the following
\begin{align*}
\norm{(G_\nu)_\nu}{X}^r&=\norm{\left(\sum_{k=-\infty}^\infty2^{-|\nu-k|\delta}g_k\right)_\nu}{X}^r=\norm{\left(\sum_{l\in\Z}2^{-|l|\delta}g_{\nu+l}\right)_\nu}{X}^r\\
&\lesssim\sum_{l\in\Z}2^{-|l|r\delta}\norm{(g_{\nu+l})_\nu}{X}^r\leq c\norm{(g_\nu)_\nu}{X}^r.
\end{align*}
Now taking the power $1/r$ yields the desired estimate.
\end{proof}
The spaces $\ellpq$ and $\ellqp$ are quasi-Banach spaces which fulfill the conditions of Lemma \ref{lem:convolutionSeq}.
Therefore, we obtain the following
\begin{lem}\label{_LM_Lemma2}
Let $p,q\in\P$ and $\delta>0$. Let $(g_k)_{k\in\Z}$ be a sequence of non-negative measurable functions on $\R^n$ and denote
\begin{align*}
 G_\nu(x)=\sum_{k\in\Z}2^{-|\nu-k|\delta}g_k(x),\qquad x\in\R^n,\,\nu\in\Z.
\end{align*}
Then there exist a constants $C_1,C_2>0$, depending on $\p,\q$ and $\delta$, such that
\begin{align*}
\norm{G_\nu}{\ellqp}&\leq C_1\norm{g_k}{\ellqp}\quad\text{and}\\
\norm{G_\nu}{\ellpq}&\leq C_2\norm{g_k}{\ellpq}.
\end{align*}
\end{lem}
\begin{rem}
Of course, Lemma \ref{_LM_Lemma2} holds true also if the indices $k$ and $\nu$ run only over natural numbers.
\end{rem}
Since the maximal operator is in general not bounded on $\ellqp$ (see \cite[Example 4.1]{AlmeidaHasto}) we need a replacement for that. 
It turned out that a convolution with radial decreasing functions fits very well into the scheme. A careful evaluation of the proof in \cite[Lemma 4.7]{AlmeidaHasto} together with Lemma \ref{Lem:eta1} gives us the following convolution inequality. 
\begin{lem}\label{Lemma_4}
Let $p,q\in\Plog$ with $\p\geq1$ and let $\eta_{\nu,m}(x)=2^{n\nu}(1+2^\nu|x|)^{-m}$. 
For all $m>n+c_{log}(1/q)$ there exists a constant $c>0$ such that for all sequences $(f_j)_{j\in\N_0}\in\ellqp$ it holds
\begin{align*}
\norm{(\eta_{\nu,m}\ast f_\nu)_{\nu\in\N_0}}{\ellqp}\leq c\norm{(f_j)_{j\in\N_0}}{\ellqp}.
\end{align*}
\end{lem}
The last technical lemma is overtaken literally from \cite{Rychkov}.
\begin{lem}[\cite{Rychkov}, Lemma 3]\label{_LM_Lemma3}
Let $0<r\leq1$ and let $(\gamma_\nu)_{\nu\in\N_0}$,
$(\beta_\nu)_{\nu\in\N_0}$ be two sequences taking values in  $(0,\infty)$. Assume that for some $N^0\in\N_0$,
\begin{align}\label{_LM_Lemma3N_0}
\limsup_{\nu\to\infty}\frac{\gamma_\nu}{2^{\nu N^0}}<\infty.
\end{align}
Furthermore, we assume that for any $N\in\N$
\begin{align*}
\gamma_\nu\leq C_N\sum_{k=0}^\infty2^{-kN}\beta_{k+\nu}\gamma_{k+\nu}^{1-r},\quad \nu\in\N_0,\quad C_N<\infty
\end{align*}
holds, then for any $N\in\N$
\begin{align*}
\gamma_\nu^r\leq C_N\sum_{k=0}^\infty2^{-kNr}\beta_{k+\nu},\quad \nu\in\N_0
\end{align*}
holds with the same constants $C_N$.
\end{lem}

\subsubsection{Comparison of different Peetre maximal operators}\label{Sec:Comp}
In this subsection we present an inequality between different
Peetre maximal operators. Let us recall the notation given before Theorem \ref{Theorem:_LM_abstrakt}.
For two given functions $\psi_0,\psi_1\in\Sn$ we define
\begin{align*}
\psi_j(x)=\psi_1(2^{-j+1}x),\quad\text{for $x\in\R^n$ and $j\in\N$.}
\end{align*}
Furthermore, for all $j\in\N_0$ we write $\Psi_j=\hat{\psi_j}$ and
in an analogous manner we define $\Phi_j$ from two starting functions $\phi_0,\phi_1\in\Sn$.
Using this notation we are ready to formulate the theorem.

\begin{thm}\label{Theorem:Comp_PO_s}
Let $\vek{w}=(w_j)_{j\in\N_0}\in\mgk$, $p,q\in\P$ and $a>0$. Moreover, let $R\in\N_0$ with $R>\alpha_2$,
\begin{align}
D^\beta\psi_1(0)=0,\quad0\leq|\beta|<R\label{Comp_PO_sMomentCond}
\end{align}
and for some $\varepsilon>0$
\begin{align}
|\phi_0(x)|&>0\quad\text{on}\quad\{x\in\R^n:|x|<\varepsilon\}\label{_LM_Tauber1}\\
|\phi_1(x)|&>0\quad\text{on}\quad\{x\in\R^n:\varepsilon/2<|x|<2\varepsilon\}\label{_LM_Tauber2},
\end{align}
then
\begin{align*}
\norm{(\Psi_k^*f)_a w_k}{\ellqp}\leq c\norm{(\Phi_k^*f)_a w_k}{\ellqp}
\end{align*}
holds for every $f\in\SSn$.
\end{thm}
\begin{rem}
	Observe that there are no restrictions on $a>0$ and $p,q\in\P$ in the theorem above.
\end{rem}
\begin{proof}
We have the fixed resolution of unity from Definition \ref{dfn_decomp} and define the 
sequence of functions $(\lambda_j)_{j\in\N_0}$ by
\begin{align*}
\lambda_j(x)=\frac{\varphi_j\left(\frac{2x}{\varepsilon}\right)}{\phi_j(x)}.
\end{align*}
It follows from the Tauberian conditions \eqref{_LM_Tauber1} and \eqref{_LM_Tauber2} that they satisfy
\begin{equation}
\sum_{j=0}^\infty\lambda_j(x)\phi_j(x)=1,\quad x\in\R^n,\label{Comp_PO_s1}
\end{equation}
\begin{equation}
\lambda_j(x)=\lambda_1(2^{-j+1}x),\quad x\in\R^n,\quad j\in \N,\label{Comp_PO_s1,2}\
\end{equation}
and
\begin{equation}
\supp\lambda_0\subset\{x\in\R^n:|x|\leq\varepsilon\}\quad\text{and}\quad
\supp\lambda_1\subset\{x\in\R^n:\varepsilon/2\leq|x|\leq2\varepsilon\}.\label{Comp_PO_s1,5}
\end{equation}
Furthermore, we denote $\Lambda_k=\hat{\lambda_k}$ for
$k\in\N_0$ and obtain together with \eqref{Comp_PO_s1} the following
identities (convergence in $\SSn$)
\begin{align}\label{Comp_PO_s2,5}
f=\sum_{k=0}^\infty\Lambda_k\ast\Phi_k\ast f,\qquad\Psi_\nu\ast f=\sum_{k=0}^\infty\Psi_\nu\ast\Lambda_k\ast\Phi_k\ast f.
\end{align}
We have
\begin{align}
|(\Psi_\nu\ast\Lambda_k\ast\Phi_k\ast f)(y)|
&\leq\int_{\R^n}|(\Psi_\nu\ast\Lambda_k)(z)||(\Phi_k\ast f)(y-z)|dz\notag\\
&\leq(\Phi_k^*f)_a(y)\int_{\R^n}|(\Psi_\nu\ast\Lambda_k)(z)|(1+|2^kz|^a)dz\label{Comp_PO_s3}\\
&=:(\Phi_k^*f)_a(y)I_{\nu,k},\notag
\intertext{where}
I_{\nu,k}&:=\int_{\R^n}|(\Psi_\nu\ast\Lambda_k)(z)|(1+|2^kz|^a)dz\notag.
\end{align}

According to Lemma \ref{_LM_Lemma1} we get
\begin{align}
I_{\nu,k}\leq c
\begin{cases}\label{Comp_PO_s2}
2^{(k-\nu)R},\qquad&\text{ } k\leq\nu,\\
2^{(\nu-k)(a+1+|\alpha_1|)},\qquad&\text{ }\nu\leq k.
\end{cases}
\end{align}
Namely, we have for $1\leq k<\nu$ with the change of variables $2^kz\mapsto z$
\begin{align*}
I_{\nu,k}&=2^{-n}\int_{\R^n}|(\Psi_{\nu-k}\ast\Lambda_1(\cdot/2))(z)|(1+|z|^a)dz\\ 
&\leq c\sup_{z\in\R^n}|(\Psi_{\nu-k}\ast\Lambda_1(\cdot/2))(z)|(1+|z|)^{a+n+1}\leq c2^{(k-\nu)R}.
\end{align*}
Similarly, we get for $1\leq\nu<k$ with the substitution $2^\nu z\mapsto z$
\begin{align*}
I_{\nu,k}&=2^{-n}\int_{\R^n}|(\Psi_1(\cdot/2)\ast\Lambda_{k-\nu})(z)|(1+|2^{k-\nu}z|^a)dz\\
&\leq c2^{(\nu-k)(M-a)}.
\end{align*}
$M$ can be taken arbitrarily large because $\Lambda_1$ has infinitely many vanishing moments. Taking
$M=2a+|\alpha_1|+1$ we derive
\eqref{Comp_PO_s2} for the cases $k,\nu\geq1$ with $k\neq\nu$. The missing cases can be treated
separately in an analogous manner. The needed moment conditions are always satisfied by
\eqref{Comp_PO_sMomentCond} and \eqref{Comp_PO_s1,5}. The case $k=\nu=0$ is covered by the constant $c$ in \eqref{Comp_PO_s2}.

Furthermore, we have
\begin{align*}
(\Phi_k^*f)_a(y)&\leq(\Phi_k^*f)_a(x)(1+|2^k(x-y)|^a)\\
&\leq(\Phi_k^*f)_a(x)(1+|2^\nu(x-y)|^a)\max(1,2^{(k-\nu)a}).
\end{align*}
We put this into \eqref{Comp_PO_s3}
and get
\begin{align}
\sup_{y\in\R^n}\frac{|(\Psi_\nu\ast\Lambda_k\ast\Phi_k\ast f)(y)|}{1+|2^\nu(x-y)|^a}
&\leq c(\Phi_k^*f)_a(x)\begin{cases}
2^{(k-\nu)R},\qquad&\; k\leq\nu,\notag\\
2^{(\nu-k)(1+|\alpha_1|)},&\; k\geq\nu.\notag
\end{cases}
\end{align}
Multiplying both sides with $w_\nu(x)$ and using
\begin{align}
&\hspace{-5em}w_\nu(x)\leq w_k(x)\begin{cases}\notag
2^{(k-\nu)(-\alpha_2)},\qquad&\; k\leq\nu,\\
2^{(\nu-k)\alpha_1},\qquad&\; k\geq\nu,
\end{cases}
\end{align}
leads us to
\begin{align}
\sup_{y\in\R^n}\frac{|(\Psi_\nu\ast\Lambda_k\ast\Phi_k\ast f)(y)|}{1+|2^\nu(x-y)|^a}w_\nu(x)
&\leq c(\Phi_k^*f)_a(x)w_k(x)\begin{cases}\notag
2^{(k-\nu)(R-\alpha_2)},\qquad&\; k\leq\nu,\\
2^{(\nu-k)},&\; k\geq\nu.
\end{cases}
\end{align}

This inequality together with
\eqref{Comp_PO_s2,5} gives for $\delta:=\min(1,R-\alpha_2)>0$
\begin{align*}
(\Psi_\nu^*f)_a(x)w_\nu(x)\leq
c\sum_{k=0}^\infty2^{-|k-\nu|\delta}(\Phi_k^*f)_a(x)w_k(x),\quad x\in\R^n.
\end{align*}
Taking the $\ellqp$ norm and using Lemma \ref{_LM_Lemma2} yields immediately the desired result.
\end{proof}

\subsubsection{Boundedness of the Peetre maximal operator}\label{Sec:Bound}
We will present a theorem which describes the boundedness of the Peetre maximal operator. We use the same notation introduced at the
beginning of the last subsection. Especially, we have the functions $\psi_k\in\Sn$ and $\Psi_k=\hat{\psi}_k\in\Sn$ for all $k\in\N_0$.

\begin{thm}\label{Theorem:_Bo_PMO}
Let $(w_k)_{k\in\N_0}\in\mgk$, $a>0$ and $p,q\in\Plog$. For some $\varepsilon>0$ we assume $\psi_0,\psi_1\in\Sn$ with
\begin{align*}
|\psi_0|&>0\quad\text{on}\quad\{x\in\R^n:|x|<\varepsilon\},\\
|\psi_1|&>0\quad\text{on}\quad\{x\in\R^n:\varepsilon/2<|x|<2\varepsilon\}.
\end{align*}
For $a>\frac{n+c_{\log}(1/q)}{p^-}+\alpha$
\begin{align*}
\norm{(\Psi_k^*f)_aw_k}{\ellqp}\leq c\norm{(\Psi_k\ast f)w_k}{\ellqp}
\end{align*}
holds for all $f\in\SSn$.
\end{thm}
\begin{rem}
Observe that in the theorem above no moment conditions on $\psi_1$ are stated but this time there are restrictions on $a$ and $\p,\q$. 
\end{rem}
\begin{proof}
As in the last proof we find the functions $(\lambda_j)_{j\in\N_0}$ with the properties \eqref{Comp_PO_s1,2}, \eqref{Comp_PO_s1,5}
and
\begin{align*}
\sum_{k=0}^\infty\lambda_k(2^{-\nu}x)\psi_k(2^{-\nu}x)=1\quad\text{for all $\nu\in\N_0$}.
\end{align*}
Instead of \eqref{Comp_PO_s2,5} we get the identity
\begin{align}\label{_Bo__PO_1}
\Psi_\nu\ast f=\sum_{k=0}^\infty\Lambda_{k,\nu}\ast\Psi_{k,\nu}\ast\Psi_\nu\ast f,
\end{align}
where
\begin{align*}
\Lambda_{k,\nu}(\xi)=[\lambda_k(2^{-\nu}\cdot)]^\wedge(\xi)=2^{\nu n}\Lambda_{k}(2^\nu\xi)\qquad\text{for all $\nu,k\in\N_0$}.
\end{align*}
The $\Psi_{k,\nu}$ are defined similarly. For $k\geq1$ and $\nu\in\N_0$ we have $\Psi_{k,\nu}=\Psi_{k+\nu}$ and with the notation
\begin{align*}
\sigma_{k,\nu}(x)=  \begin{cases}
\psi_0(2^{-\nu}x)\qquad&\text{ if }k=0,\\
\psi_\nu(x)\qquad&\text{otherwise}
\end{cases}
\end{align*}
we get
$\psi_k(2^{-\nu}x)\psi_\nu(x)=\sigma_{k,\nu}(x)\psi_{k+\nu}(x)$. Hence, we can rewrite \eqref{_Bo__PO_1} as
\begin{align}\label{_Bo__PO_2}
\Psi_\nu\ast
f=\sum_{k=0}^\infty\Lambda_{k,\nu}\ast\hat{\sigma}_{k,\nu}\ast\Psi_{k+\nu}\ast f.
\end{align}
For $k\geq1$ we get from Lemma \ref{_LM_Lemma1}
\begin{align}\label{_Bo__PO_2,5}
|(\Lambda_{k,\nu}\ast\hat{\sigma}_{k,\nu})(z)|=2^{(\nu-1) n}|(\Lambda_{k}\ast\Psi_1(\cdot/2))(2^\nu z)|
\leq C_M2^{\nu n}\frac{2^{-kM}}{(1+|2^\nu z|^a)}
\end{align}
for all $k,\nu\in\N_0$ and arbitrary large $M\in\N$. For $k=0$ we get the estimate \eqref{_Bo__PO_2,5} by using Lemma
\ref{_LM_Lemma1} with $M=0$. This together with \eqref{_Bo__PO_2} gives us
\begin{align}\label{_Bo__PO_3}
|(\Psi_\nu\ast f)(y)|\leq C_M2^{\nu n}\sum_{k=0}^\infty\int_{\R^n}\frac{2^{-kM}}{(1+|2^\nu(y-z)|^a)}|(\Psi_{k+\nu}\ast f)(z)|dz.
\end{align}
For fixed $r\in(0,1]$ we divide both sides of \eqref{_Bo__PO_3} by $(1+|2^\nu(x-y)|^a)$
and we take the supremum with respect to $y\in\R^n$. Using the inequalities
\begin{align*}
(1+|2^\nu(y-z)|^a)(1+|2^\nu(x-y)|^a)\geq c(1+|2^\nu(x-z)|^a),\\
|(\Psi_{k+\nu}\ast f)(z)|\leq|(\Psi_{k+\nu}\ast f)(z)|^r(\Psi_{k+\nu}^* f)_a(x)^{1-r}(1+|2^{k+\nu}(x-z)|^a)^{1-r}\\
\intertext{and}
\frac{(1+|2^{k+\nu}(x-z)|^a)^{1-r}}{(1+|2^{\nu}(x-z)|^a)}\leq\frac{2^{ka}}{(1+|2^{k+\nu}(x-z)|^a)^{r}},
\end{align*}
we get
\begin{align}
\!\!(\Psi_\nu^*f)_a(x)\leq
C_M\sum_{k=0}^\infty2^{-k(M+n-a)}(\Psi_{k+\nu}^*f)_a(x)^{1-r}\!\!\!\int_{\R^n}\!\!\!
\frac{2^{(k+\nu)n}|(\Psi_{k+\nu}\ast f)(z)|^r}{(1+|2^{k+\nu}(x-z)|^a)^{r}}dz.
\end{align}
Now, we apply Lemma \ref{_LM_Lemma3} with
\begin{align*}
\gamma_\nu=(\Psi_\nu^*f)_a(x),\qquad\beta_\nu=
\int_{\R^n}\frac{2^{\nu n}|(\Psi_{\nu}\ast f)(z)|^r}{(1+|2^{\nu}(x-z)|^a)^{r}}dz,\quad\nu\in\N_0
\end{align*}
$N=M+n-a$, $C_N=C_M+n-a$ and $N^0$ in \eqref{_LM_Lemma3N_0} equals the order of the distribution $f\in\SSn$.

By Lemma \ref{_LM_Lemma3} we obtain for every $N\in\N$, $x\in\R^n$ and $\nu\in\N_0$
\begin{align}\label{_Bo__PO_4}
(\Psi_\nu^*f)_a(x)^r\leq 
C_N\sum_{k=0}^\infty2^{-kNr}\int_{\R^n}\frac{2^{(k+\nu)n}|(\Psi_{k+\nu}\ast f)(z)|^r}{(1+|2^{k+\nu}(x-z)|^a)^{r}}dz
\end{align}
provided that $(\Psi_\nu^*f)_a(x)<\infty$.

Since $f\in\SSn$, we see that $(\Psi_\nu^*f)_a(x)<\infty$ for all $x\in\R^n$ and all $\nu\in\N_0$ at 
least if $a>N^0$, where $N^0$ is the order of the distribution. Thus we have \eqref{_Bo__PO_4} with $C_N$ 
independent of $f\in\SSn$ for $a\geq N^0$ and therefore with $C_N=C_{N,f}$ for all $a>0$ (the right side of 
\eqref{_Bo__PO_4} decreases as $a$ increases). One can easily check that \eqref{_Bo__PO_4} with $C_N=C_{N,f}$ 
implies that if for some $a>0$ the right side of \eqref{_Bo__PO_4} is finite, then $(\Psi_\nu^*f)_a(x)<\infty$. 
Now, repeating the above argument resurrects the independence of $C_N$.
If the right side of \eqref{_Bo__PO_4} is infinite, there is nothing to prove. More exhaustive arguments of this type have been used in \cite{UllrichLM} 
and \cite{Scharf}. 

We point out that \eqref{_Bo__PO_4} holds also for $r>1$, where the proof is much simpler. We only have
to take \eqref{_Bo__PO_3} with $a+n$ instead of $a$, divide both sides by $(1+|2^\nu(x-y)|^a)$
and apply H{\"o}lder's inequality with respect to $k$ and then $z$.

Multiplying \eqref{_Bo__PO_4} by $w_\nu(x)^r$ we derive with the properties of our weight sequence
\begin{align}\label{ProofEnde2}
\!\!\!(\Psi_\nu^*f)_a(x)^rw_\nu(x)^r\leq
C'_N\sum_{k=0}^\infty2^{-k(N+\alpha_1)r}\!\!\!
\int_{\R^n}\!\!\!\frac{2^{(k+\nu)n}|(\Psi_{k+\nu}\ast f)(z)|^rw_{k+\nu}(z)^r}{(1+|2^{k+\nu}(x-z)|^{a-\alpha})^{r}}dz,
\end{align}
for all $x\in\R^n$, $\nu\in\N_0$ and all $N\in\N$.

Now, we choose $r=p^-$ and we have $r(a-\alpha)>n+c_{\log}(1/q)$. We denote $g^r_{k+\nu}(z)=|(\Psi_{k+\nu}\ast f)(z)|^rw_{k+\nu}(z)^r$ then we can 
rewrite \eqref{ProofEnde2} by
\begin{align}\label{ProofEnde2,5}
\!\!\!(\Psi_\nu^*f)_a(x)^rw_\nu(x)^r\leq
C'_N\sum_{l=\nu}^\infty2^{-(l-\nu)(N+\alpha_1)r}\left(g_l^r\ast\eta_{l,r(a-\alpha)}\right)(x).
\end{align}
For fixed $N>0$ with $\delta=N+\alpha_1>0$ we apply the $\ell_{\frac\q r}(L_{\frac\p r})$ norm and derive from \eqref{ProofEnde2,5}
\begin{align*}
\norm{(\Psi_k^*f)_a^rw_k^r}{\ell_{\q/r}(L_{\p/r})}\leq C_N\norm{\sum_{l=\nu}^\infty2^{-(l-\nu)\delta}\left(g_l^r\ast\eta_{l,r(a-\alpha)}\right)}{\ell_{\q/r}(L_{\p/r})}.
\end{align*}
Now application of Lemma \ref{_LM_Lemma2} and Lemma \ref{Lemma_4} ($r(a-\alpha)>n+c_{\log}(1/q)$) on the formula above give us
\begin{align*}
 \norm{(\Psi_k^*f)_a^r(\cdot)w_k^r(\cdot)}{\ell_{\q/r}(L_{\p/r})}\leq C'_N\norm{|(\Psi_{\nu}\ast f)(\cdot)|^rw_{\nu}(\cdot)^r}{\ell_{\q/r}(L_{\p/r})}
\end{align*}
which proves the theorem.
\end{proof}

\subsection{Local means characterization of $\Bspqpunkt$ and $\Fspqpunkt$}\label{SecLocsx}
In this section we reformulate the local means characterization for $\Bwpqpunkt$ from above and for $\Fwpqpunkt$ from Corollary 4.7 in \cite{Kempka09} in terms of variable smoothness. If we have a variable smoothness function $s\in C_{loc}^{\log}(\R^n)$ given, then $w_j(x)=2^{js(x)}$ defines an admissible weight sequence $\vek{w}\in\mgk$ with $\alpha_1=s^-$, $\alpha_2=s^+$ and $\alpha=c_{\log}(s)$, cf. Remark \ref{rem:2mlsx}. Here, we denote by $c_{\log}(s)$ the constant in \eqref{eq:defclog} for $\s$. 
\begin{thm}\label{thm:sx:loc1}
Let $p,q\in\Plog$ ($p^+,q^+<\infty$ in the F-case) and $s\in C_{loc}^{\log}(\R^n)$. Further let $a>0$, $R\in\N_0$ with $R>s^+$
and let $\psi_0,\psi_1$ belong to $\Sn$ with
\begin{align*}
D^\beta\psi_1(0)=0,\quad\text{ for }0\leq|\beta|< R,
\end{align*}
and
\begin{align*}
|\psi_0(x)|&>0\quad\text{on}\quad\{x\in\R^n:|x|<\varepsilon\}\\
|\psi_1(x)|&>0\quad\text{on}\quad\{x\in\R^n:\varepsilon/2<|x|<2\varepsilon\}
\end{align*}
for some $\varepsilon>0$.
\begin{enumerate}
	\item For $a>\frac{n+c_{\log}(1/q)}{p^-}+c_{\log}(s)$ and all $f\in\SSn$ we have
\begin{align*}
\norm{f}{\Bspqpunkt}\approx\norm{2^{ks(\cdot)}(\Psi_k\ast f)}{\ellqp}\approx\norm{2^{ks(\cdot)}(\Psi_k^*f)_a}{\ellqp}.
\end{align*}
\item	For $a>\frac{n}{\min(p^-,q^-)}+c_{\log}(s)$ and all $f\in\SSn$ we have 
\begin{align*}
\norm{f}{\Fspqpunkt}\approx\norm{2^{ks(\cdot)}(\Psi_k\ast f)}{\ellpq}\approx\norm{2^{ks(\cdot)}(\Psi_k^*f)_a}{\ellpq}.
\end{align*}
\end{enumerate}
\end{thm}
\begin{rem}
During the referee process of this work, there appeared in \cite{Drihem} a characterization by local means and a characterization by atoms for $\Bspqpunkt$. The author moved the smoothness sequence $2^{ks(\cdot)}$ into the Peetre maximal operator \eqref{_PO_modifiziert} and modified it to 
\begin{align*}
	(\Psi_k^*2^{k\s}f)_a(x)=\sup_{y\in\R^n}\frac{2^{ks(y)}|(\Psi_k\ast f)(y)|}{1+|2^k(y-x)|^a}.
\end{align*} 
For this modified Peetre maximal operator he obtained in \cite[Theorem 2]{Drihem} an equivalence of the norms similar to our Theorem \ref{thm:sx:loc1} for $\Bspqpunkt$. The advantage of his method is that the condition on $a>0$ weakens to $a>\frac{n}{p^-}$.
\end{rem}

\section{Ball means of differences}

This section is devoted to the characterization of Besov and Triebel-Lizorkin spaces $\Bspqpunkt$
and $\Fspqpunkt$ by ball means of differences. In the case of constant indices $p,q$ and $s$,
this is a classical part of the theory of function spaces. We refer especially to \cite[Section 2.5]{Triebel1}
and references given there. It turns out that, under the restriction
\begin{equation}\label{eq:cond1}
s>\sigma_{p}=n\left(\frac{1}{\min(p,1)}-1\right)
\end{equation}
in the $B$-case and
\begin{equation}\label{eq:cond2}
s>\sigma_{p,q}=n\left(\frac{1}{\min(p,q,1)}-1\right)
\end{equation}
in the $F$-case, Besov and Triebel-Lizorkin spaces with constant indices may be characterized by expressions involving only the differences
of the function values without any use of Fourier analysis. This was complemented in \cite{Connie1}
and \cite{Connie2} by showing that these conditions are also indispensable. Of course, we are limited by \eqref{eq:cond1}
and \eqref{eq:cond2} also in the case of variable exponents.

The characterization by (local means of) differences for 2-microlocal spaces with constant $p,q>1$ was given by Besov \cite{Bes1} and a similar characterization for Besov spaces with $p=q=\infty$ and the special weight sequence from \eqref{tmlweights} was given by Seuret and Levy V\'eh\'el in \cite{VehelSeuret03}. 
We refer to \cite{FarLeo} and \cite{Kal} for the treatment of spaces of generalized smoothness.

Our approach follows essentially \cite{Triebel1} with some modifications described in \cite{U}.
The main obstacle on this way is the unboundedness of the maximal operator in the frame of $\ellpq$ and $\ellqp$
spaces, cf. \cite[Section 5]{DHR}  and \cite[Example 4.1]{AlmeidaHasto}. This is circumvented by the use of
convolution with radial functions in the sense of \cite{DHR} and \cite{AlmeidaHasto} together with
a certain bootstrapping argument, which shall be described in detail below. 

The plan of this part of the work is as follows.
First we give in Section \ref{Sec:Diff:Not} the necessary notation. 
We state the main assertions of this part in Section \ref{Sec:Diff:Thm}.
Then we prove in Section \ref{Sec:Diff:First} a certain preliminary version of these assertions.
In Section \ref{Sec:Diff:q} we prove a characterization by ball means of differences for spaces with $q\in(0,\infty]$ constant (where the maximal operator
is bounded) and use this together with our preliminary characterization from Section \ref{Sec:Diff:First}
to conclude the proof. Finally, in Section \ref{sec:2mldiff} we will present the ball means of differences characterization also for the 2-microlocal function spaces $\Bwpqpunkt$ and $\Fwpqpunkt$ and in Section \ref{Sec:Lemmas} we present separately some useful Lemmas, not to
disturb the main proofs of this part.

\subsection{Notation} \label{Sec:Diff:Not}

Let $f$ be a function on $\R^n$ and let $h\in\R^n$. Then we define
$$
 \Delta^1_h f(x)=f(x+h)-f(x),\quad x\in \R^n.
$$
The higher order differences are defined inductively by
$$
\Delta^M_h f(x)=\Delta^1_h(\Delta^{M-1}_h f)(x), \quad M=2,3,\dots
$$
This definition also allows a direct formula
\begin{equation}\label{eq:dif3}
\Delta^M_h f(x):=\sum_{j=0}^M (-1)^j\binom{M}{j}f(x+(M-j)h).
\end{equation}

By \emph{ball means of differences} we mean the quantity
$$
d^M_{t}f(x)=t^{-n}\int_{|h|\le t}|\Delta^M_{h}f(x)| dh = \int_B |\Delta^M_{th}f(x)|dh,
$$
where $B=\{y\in\R^n:|y|<1\}$ is the unit ball of $\R^n$, $t>0$ is a real number and $M$ is a natural number.

Let us now introduce the (quasi-)norms, which shall be the main subject of our study. We define
\begin{align}\label{eq:def:F1}
\begin{split}
	\|f|\Fspqpunkt\|^*&:=\|f|L_{\p}(\R^n)\|\\
&\qquad+\left\|\left(\int_0^\infty t^{-s(x)q(x)}
\left(d^M_{t}f(x)\right)^{q(x)}\frac{dt}{t}\right)^{1/q(x)}\biggl|L_{p(\cdot)}(\R^n)\right\|
\end{split}
\end{align}
and its partially discretized counterpart
\begin{align*}
\|f|\Fspqpunkt\|^{**}&:=\|f|L_{\p}(\R^n)\|\\
&\quad+\left\|\left(\sum_{k=-\infty}^\infty 2^{ks(x)q(x)}\left(d^M_{2^{-k}}f(x)\right)^{q(x)}\right)^{1/q(x)}
\biggl|L_{p(\cdot)}(\R^n)\right\|\\
&=\|f|L_{\p}(\R^n)\|+\left\|\left(2^{ks(x)}d^M_{2^{-k}}f(x)\right)_{k=-\infty}^\infty\biggl|\ellpq\right\|.
\end{align*}
The norm $\|f|\Fspqpunkt\|^{**}$ admits a direct counterpart also for Besov spaces, namely
\begin{equation}\label{eq:def:B1}
\|f|\Bspqpunkt\|^{**}:=\|f|L_{\p}(\R^n)\|+\left\|\left(2^{ks(x)}d^M_{2^{-k}}f(x)\right)_{k=-\infty}^\infty|\ellqp\right\|.
\end{equation}

Finally, we shall use as a technical tool also the analogues of \eqref{eq:def:F1}--\eqref{eq:def:B1}
with the integration over $t$ restricted to $0<t<1$. This leads to the following expressions
\begin{align*}
\|f|\Fspqpunkt\|_1^*&:=\|f|L_{\p}(\R^n)\|\\
&\qquad+\left\|\left(\int_0^1 t^{-s(x)q(x)}
\left(d^M_{t}f(x)\right)^{q(x)}\frac{dt}{t}\right)^{1/q(x)}\biggl|L_{p(\cdot)}(\R^n)\right\|,\\
\|f|\Fspqpunkt\|^{**}_1&:=\|f|L_{\p}(\R^n)\|\\
&\qquad+\left\|\left(\sum_{k=0}^\infty 2^{ks(x)q(x)}\left(d^M_{2^{-k}}f(x)\right)^{q(x)}\right)^{1/q(x)}
\biggl|L_{p(\cdot)}(\R^n)\right\|\\
&=\|f|L_{\p}(\R^n)\|+\left\|\left(2^{ks(x)} d^M_{2^{-k}}f(x)\right)_{k=0}^\infty\biggl|\ellpq\right\|,\\
\|f|\Bspqpunkt\|_1^{**}&:=\|f|L_{\p}(\R^n)\|+\left\|\left(2^{ks(x)}d^M_{2^{-k}}f(x)\right)_{k=0}^\infty|\ellqp\right\|.
\end{align*}

\subsection{Main Theorem}\label{Sec:Diff:Thm}

Using the notation introduced above, we may now state the main result of this section.

\begin{thm}\label{thm:dif}
(i) Let $p,q\in\Plog$ with $p^+,q^+<\infty$ and $s\in C^{\rm log}_{loc}(\R^n)$. Let $M\in\N$ with $M>s^+$ and let
\begin{equation}\label{eq:thm:cond1}
s^->\sigma_{p^-,q^-}\cdot\left[1+\frac{c_{\rm log}(s)}{n}\cdot \min(p^-,q^-)\right].
\end{equation}
Then
$$
\Fspqpunkt=\{f\in \Lpp\cap\SSn:\|f|\Fspqpunkt\|^*<\infty\}
$$
and $\|\cdot|\Fspqpunkt\|$ and $\|\cdot|\Fspqpunkt\|^*$ are equivalent on $\Fspqpunkt$.
The same holds for $\|f|\Fspqpunkt\|^{**}$.

(ii) Let $p,q\in\Plog$ and $s\in C^{\rm log}_{loc}(\R^n)$. Let $M\in\N$ with $M>s^+$ and let
\begin{equation}\label{eq:thm:cond2}
s^->\sigma_{p^-}\cdot \left[1+\frac{c_{\rm log}(1/q)}{n}+\frac{c_{\rm log}(s)}{n}\cdot p^-\right].
\end{equation}
Then
$$
\Bspqpunkt=\{f\in \Lpp\cap\SSn:\|f|\Bspqpunkt\|^{**}<\infty\}
$$
and $\|\cdot|\Bspqpunkt\|$ and $\|\cdot|\Bspqpunkt\|^{**}$ are equivalent on $\Bspqpunkt$.
\end{thm}
\begin{rem}
Let us comment on the rather technical conditions \eqref{eq:thm:cond1} and \eqref{eq:thm:cond2}.
\begin{itemize}
\item If $\min(p^-,q^-)\ge 1$, then \eqref{eq:thm:cond1} becomes just $s^->0$. Furthermore, if $p$, $q$
and $s$ are constant functions, then \eqref{eq:thm:cond1} coincides with \eqref{eq:cond2}.
\item If $p^-\ge 1$, then \eqref{eq:thm:cond2} reduces also to $s^->0$ and in the case
of constant exponents we again recover \eqref{eq:cond1}.
\end{itemize}
\end{rem}
As indicated already above, the proof is divided into several parts.

\subsection{Preliminary version of Theorem \ref{thm:dif}}\label{Sec:Diff:First}

This subsection contains a preliminary version of Theorem \ref{thm:dif} (Lemma \ref{Lem:First}).
Its proof represents the heart of the proof of Theorem \ref{thm:dif}. For better lucidity, it is again divided into more parts.

\begin{lem}\label{Lem:First}
Under the conditions of Theorem \ref{thm:dif}, the following estimates hold for all $f\in \Lpp\cap\SSn$:
\begin{align}
\label{eq:todo1}\|f|\Fspqpunkt\|^{*}&\approx \|f|\Fspqpunkt\|^{**},\\
\label{eq:todo2}\|f|\Fspqpunkt\|_1^{*}&\approx \|f|\Fspqpunkt\|_1^{**},\\
\label{eq:todo3}\|f|\Fspqpunkt\|_1^{**}&\lesssim \|f|\Fspqpunkt\|\lesssim \|f|\Fspqpunkt\|^{**},\\
\label{eq:todo4}\|f|\Bspqpunkt\|_1^{**}&\lesssim \|f|\Bspqpunkt\|\lesssim \|f|\Bspqpunkt\|^{**}.
\end{align}
\end{lem}

\begin{proof}

\emph {Part I.} First we prove \eqref{eq:todo1} and \eqref{eq:todo2}.
We discretize the inner part of $\|\cdot\|^*$ and obtain
{\allowdisplaybreaks
\begin{align}
\notag \biggl[\int_0^\infty& t^{-s(x)q(x)}\left(\int_B |\Delta^M_{th}f(x)|dh\right)^{q(x)}\frac{dt}{t}\biggr]^{1/q(x)}\\
\label{eq:diff:F1}&= \biggl[\int_0^\infty t^{-s(x)q(x)}\left(t^{-n}\int_{tB}|\Delta_\varkappa^M f(x)|d\varkappa\right)^{q(x)}\frac{dt}{t}\biggr]^{1/q(x)}\\
\notag &= \biggl[\sum_{k=-\infty}^\infty \int_{2^{-k-1}}^{2^{-k}} t^{-s(x)q(x)}\left(t^{-n}\int_{tB}|\Delta_\varkappa^M f(x)|d\varkappa\right)^{q(x)}\frac{dt}{t}\biggr]^{1/q(x)}.
\end{align}
If $2^{-k-1}\le t\le 2^{-k}$, then $2^{ks(x)q(x)}\le t^{-s(x)q(x)}\le 2^{(k+1)s(x)q(x)}$ and
$$
2^{kn}\int_{2^{-(k+1)}B}|\Delta_\varkappa^M f(x)|d\varkappa\lesssim
t^{-n}\int_{tB}|\Delta_\varkappa^M f(x)|d\varkappa \lesssim
2^{(k+1)n}\int_{2^{-k}B}|\Delta_\varkappa^M f(x)|d\varkappa.
$$
Plugging these estimates into \eqref{eq:diff:F1}, we may further estimate  
\begin{align*}
\biggl[\int_0^\infty& t^{-s(x)q(x)}\left(\int_B |\Delta^M_{th}f(x)|dh\right)^{q(x)}\frac{dt}{t}\biggr]^{1/q(x)}\\
&\lesssim \biggl[\sum_{k=-\infty}^\infty 2^{(k+1)s(x)q(x)}\left(2^{kn}\int_{2^{-k}B}|\Delta_\varkappa^M f(x)|d\varkappa\right)^{q(x)}\biggr]^{1/q(x)}\\
&\lesssim\biggl[\sum_{k=-\infty}^\infty 2^{ks(x)q(x)}\left(\int_{B}|\Delta_{2^{-k}\varkappa}^M f(x)|d\varkappa\right)^{q(x)}\biggr]^{1/q(x)}.
\end{align*}}
The estimate from below follows in the same manner.
Finally, the proof of \eqref{eq:todo2} is almost the same.

\emph{Part II.} This part is devoted to the proof of the left hand side of \eqref{eq:todo3}. It is divided into
several steps to make the presentation clearer. 

\emph{Step 1.} First, we point out that the estimate
$$
\norm{f}{L_\p(\R^n)}\lesssim \norm{f}{\Bspqpunkt}
$$
follows from the characterization of $\Bspqpunkt$ in terms of Nikol'skij representations (cf. Theorem 8.1 of \cite{AlmeidaHasto}).
We refer also to Remark 2.5.3/1 in \cite{Triebel1}. The extension to $F$-spaces is then given by the simple embedding
$$
\norm{f}{L_\p(\R^n)}\lesssim \norm{f}{B^{s(\cdot)-\varepsilon}_{p(\cdot),p(\cdot)}(\R^n)}\lesssim
\norm{f}{\Fspqpunkt}
$$
with $\varepsilon>0$ chosen small enough.

\emph{Step 2.} Let $(\varphi_j)_{j\in\N_0}$ be the functions used in Definition \ref{def:BFpunkt}.
We use the decomposition
$$
f=\sum_{l=-\infty}^\infty f_{(k+l)},\quad k\in\Z,
$$
where $f_{(k+l)}=(\varphi_{k+l}\hat f)^\vee$, or $=0$ if $k+l<0$ and get
\begin{align*}
(\clubsuit)&:=\sum_{k=0}^\infty 2^{ks(x)q(x)}\left(\int_{B}|\Delta_{2^{-k}h}^M f(x)|dh\right)^{q(x)}\\
&=\sum_{k=0}^\infty 2^{ks(x)q(x)}\left(\int_{B}|\Delta_{2^{-k}h}^M \left(\sum_{l=-\infty}^\infty f_{(k+l)}\right)(x)|dh\right)^{q(x)}.
\end{align*}
If $q(x)\le 1$ then we proceed further
\begin{align*}
(\clubsuit)&\le \sum_{k=0}^\infty 2^{ks(x)q(x)}\left(\int_{B}\sum_{l=-\infty}^\infty|\Delta_{2^{-k}h}^M f_{(k+l)}(x)|dh\right)^{q(x)}\\
&\le \sum_{k=0}^\infty \sum_{l=-\infty}^\infty 2^{ks(x)q(x)}\left(\int_{B}|\Delta_{2^{-k}h}^M f_{(k+l)}(x)|dh\right)^{q(x)}.
\end{align*}
If $q(x)>1$, we use Minkowski's inequality
\begin{align*}
(\clubsuit)^{1/q(x)}&\le \left(\sum_{k=0}^\infty 2^{ks(x)q(x)}\left(\int_{B}\sum_{l=-\infty}^\infty|\Delta_{2^{-k}h}^M f_{(k+l)}(x)|dh\right)^{q(x)}\right)
^{1/q(x)}\\
&\le \sum_{l=-\infty}^\infty \left(\sum_{k=0}^\infty 2^{ks(x)q(x)}\left(\int_{B}|\Delta_{2^{-k}h}^M f_{(k+l)}(x)|dh\right)^{q(x)}\right)^{1/q(x)}.
\end{align*}

We split in both cases
\begin{equation}\label{eq:proof1}
\sum_{l=-\infty}^\infty\dots=I+II=\sum_{l=-\infty}^0 \dots + \sum_{l=1}^\infty \dots
\end{equation}

\emph{Step 3.} We estimate the first summand with $l\le 0$.

We use Lemma \ref{lem:dif_Peetre} in the form
$$
|\Delta^M_h f_{(k+l)}(x)|\le C\max(1,|bh|^a)\cdot\min(1,|bh|^M)P_{b,a}f_{(k+l)}(x),
$$
where $a>0$ is arbitrary, $b=2^{k+l}$ and
$$
P_{b,a}f(x)=\sup_{z\in\R^n}\frac{|f(x-z)|}{1+|bz|^a}.
$$
Furthermore, we use this estimate with $2^{-k}h$ instead of $h$. We obtain
\begin{align}
\notag\int_B|\Delta^M_{2^{-k}h}f_{(k+l)}(x)|dh&\lesssim \int_B \max(1,|b2^{-k}h|^a)\cdot
\min(1,|b2^{-k}h|^M) P_{b,a}f_{(k+l)}(x)dh\\
&\label{eq:dif1}\lesssim 2^{lM}P_{2^{k+l},a}f_{(k+l)}(x).
\end{align}
The last inequality follows from $\max(1,|b2^{-k}h|^a)\le 1$ (recall that $l\le 0$ and $|h|\le 1$) and
$\min(1,|b2^{-k}h|^M)\le 2^{lM}$.

If $q(x)\le 1$, we estimate the first sum in \eqref{eq:proof1}
\begin{align*}
I&\le\sum_{l=-\infty}^0 \sum_{k=0}^\infty 2^{ks(x)q(x)}\left(\int_{B}|\Delta_{2^{-k}h}^M f_{(k+l)}(x)|dh\right)^{q(x)}\\
&\lesssim \sum_{l=-\infty}^0 \sum_{k=0}^\infty 2^{ks(x)q(x)}\left(2^{lM}P_{2^{k+l},a}f_{(k+l)}(x)\right)^{q(x)}\\
&= \sum_{l=-\infty}^0 2^{l(M-s(x))q(x)}\sum_{k=0}^\infty 2^{(k+l)s(x)q(x)}P^{q(x)}_{2^{k+l},a}f_{(k+l)}(x)\\
&\approx \sum_{k=0}^\infty 2^{ks(x)q(x)}P^{q(x)}_{2^{k},a}f_{(k)}(x),
\end{align*}
where the last estimate makes use of $M>s^+$, $q^->0$ and the fact that $f_{(k+l)}=0$ for $k+l<0.$

If $q(x)>1$, we proceed in a similar way to obtain
\begin{align*}
I^{1/q(x)}&\le \sum_{l=-\infty}^0 \left(\sum_{k=0}^\infty 2^{ks(x)q(x)}\left(\int_{B}|\Delta_{2^{-k}h}^M f_{(k+l)}(x)|dh\right)^{q(x)}\right)^{1/{q(x)}}\\
&\lesssim \sum_{l=-\infty}^0 \left(\sum_{k=0}^\infty 2^{ks(x)q(x)}\left(2^{lM}P_{2^{k+l},a}f_{(k+l)}(x)\right)^{q(x)}\right)^{1/{q(x)}}\\
&= \sum_{l=-\infty}^0 2^{l(M-s(x))}\left(\sum_{k=0}^\infty 2^{(k+l)s(x)q(x)}P^{q(x)}_{2^{k+l},a}f_{(k+l)}(x)\right)^{1/{q(x)}}\\
&\lesssim \left(\sum_{k=0}^\infty 2^{ks(x)q(x)}P^{q(x)}_{2^{k},a}f_{(k)}(x)\right)^{1/q(x)}.
\end{align*}
We have used in the last estimate again $M>s^+$ and the definition of $f_{(k+l)}$.

Hence,
$$
I^{1/q(x)}\lesssim\left(\sum_{k=0}^\infty 2^{ks(x)q(x)}P^{q(x)}_{2^{k},a}f_{(k)}(x)\right)^{1/q(x)}
$$
holds for all $x\in\R^n$. 

Finally, we obtain
\begin{align}
\notag \|I^{1/q(\cdot)}|L_{p(\cdot)}(\R^n)\|&\lesssim \left\|\left(\sum_{k=0}^\infty 2^{ks(x)q(x)}P^{q(x)}_{2^{k},a}f_{(k)}(x)\right)^{1/q(x)}\bigg|L_{p(\cdot)}(\R^n)\right\|\\
\label{eq:fin1}&=\left\|\left(2^{ks(x)}P_{2^{k},a}f_{(k)}(x)\right)_{k=0}^\infty|L_{p(\cdot)}(\ell_{q(\cdot)})\right\|\\
\notag &\lesssim \left\|\left(2^{ks(\cdot)}f_{(k)}\right)_{k=0}^\infty|L_{p(\cdot)}(\ell_{q(\cdot)})\right\|=\|f|\Fspqpunkt\|,
\end{align}
where we used the boundedness of Peetre maximal operator as described in Theorem \ref{thm:sx:loc1}
for $a>0$ large enough.

\emph{Step 4.} We estimate the second summand in \eqref{eq:proof1} with $l>0$.
If $\min(p^-,q^-)>1$, then we put $\lambda=1$. Otherwise we
choose real parameters $0<\lambda<\min(p^-,q^-)$ and $a>0$ such that 
$$
a>\frac{n}{\min(p^-,q^-)}+c_{\rm log}(s)
$$
and $a(1-\lambda)<s^-$. Due to \eqref{eq:thm:cond1}, this is always possible.

We start again with estimates of the ball means of differences. We use Lemma \ref{lem:dif_Peetre} and \eqref{eq:dif3}
to obtain
\begin{align}
\notag \int_B&|\Delta^M_{2^{-k}h}f_{(k+l)}(x)|dh=\int_B|\Delta^M_{2^{-k}h}f_{(k+l)}(x)|^\lambda\cdot 
|\Delta^M_{2^{-k}h}f_{(k+l)}(x)|^{1-\lambda} dh\\
\label{eq:dif6}
\begin{split}
	&\lesssim \int_B\Bigl(\max(1,|2^{k+l}2^{-k}h|^a)\min(1,|2^{k+l}2^{-k}h|^M)P_{2^{k+l},a}f_{(k+l)}(x)\Bigr)^{1-\lambda}\cdot\\
	&\qquad\qquad \cdot |\Delta^M_{2^{-k}h}f_{(k+l)}(x)|^{\lambda}dh
\end{split}\\
\notag&\le \left(2^{la}P_{2^{k+l},a}f_{(k+l)}(x)\right)^{1-\lambda} \int_B|\Delta^M_{2^{-k}h}f_{(k+l)}(x)|^{\lambda}dh\\
\notag&\le \left(2^{la}P_{2^{k+l},a}f_{(k+l)}(x)\right)^{1-\lambda}\sum_{j=0}^M c_{j,M}\int_B|f_{(k+l)}(x+j2^{-k}h)|^\lambda dh,
\end{align}
where the constants $c_{j,M}$ are given by \eqref{eq:dif3}.

We shall deal in detail only with the term with $j=1$. The term with $j=0$ is much simpler to handle (as there the integration
over $h\in B$ immediately disappears) and this case reduces essentially to H\"older's inequality and boundedness of the Peetre
maximal operator. The terms with $2\le j \le M$ may be handled in the same way as the one with $j=1.$

We use Lemma \ref{lem:r-trick} with $r=\lambda$ in the form
$$
|f_{(k+l)}(y)|^\lambda \lesssim (\eta_{k+l,2m}*|f_{(k+l)}|^\lambda)(y),
$$
with $m>\max(n,c_{\rm log}(s))$, Lemma \ref{lem:conv} and Lemma \ref{Lem:eta1} to get
{\allowdisplaybreaks
\begin{align}
2^{ks(x)\lambda}\notag\int_B &|f_{(k+l)}(x+2^{-k}h)|^\lambda dh\\
\notag&\lesssim 2^{ks(x)\lambda}\int_B (\eta_{k+l,2m}*|f_{(k+l)}|^\lambda)(x+2^{-k}h) dh\\
&= 2^{ks(x)\lambda} ([2^{kn}\chi_{2^{-k}B} ] *\eta_{k+l,2m}*|f_{(k+l)}|^\lambda)(x)\label{eq:dif2}\\
\notag&\lesssim 2^{ks(x)\lambda}(\eta_{k,2m}*|f_{(k+l)}|^\lambda)(x)\\
\notag&\lesssim (\eta_{k,m}*|2^{ks(\cdot)}f_{(k+l)}|^\lambda)(x)\\
\notag&\le 2^{-ls^-\lambda}(\eta_{k,m}*|2^{(k+l)s(\cdot)}f_{(k+l)}|^\lambda)(x).
\end{align}}
We insert \eqref{eq:dif2} into \eqref{eq:dif6} and arrive at
\begin{align}\label{eq:dif8}
2^{ks(x)}\int_B&|\Delta^M_{2^{-k}h}f_{(k+l)}(x)|dh\\
&\notag\lesssim 2^{la(1-\lambda)-ls^-}\left(2^{(k+l)s(x)}P_{2^{k+l},a}f_{(k+l)}(x)\right)^{1-\lambda}
(\eta_{k,m}*|2^{(k+l)s(\cdot)}f_{(k+l)}|^\lambda)(x).
\end{align}

If $q(x)>1$, we proceed further with the use of H\"older's inequality
\begin{align*}
II^{1/q(x)}&\lesssim \sum_{l=1}^\infty 2^{la(1-\lambda)-ls^-}\\
&\left(\sum_{k=0}^\infty \left(2^{(k+l)s(x)}P_{2^{k+l},a}f_{(k+l)}(x)\right)^{(1-\lambda)q(x)}
(\eta_{k,m}*|2^{(k+l)s(\cdot)}f_{(k+l)}|^\lambda)^{q(x)}(x)\right)^{1/q(x)}\\
&\le \sum_{l=1}^\infty 2^{la(1-\lambda)-ls^-}
\left(\sum_{k=0}^\infty \left(2^{(k+l)s(x)}P_{2^{k+l},a}f_{(k+l)}(x)\right)^{q(x)}\right)^{(1-\lambda)/q(x)}\cdot\\
&\qquad\qquad\cdot\left(\sum_{k=0}^\infty(\eta_{k,m}*|2^{(k+l)s(\cdot)}f_{(k+l)}|^\lambda)^{q(x)/\lambda}(x)\right)^{\lambda/q(x)}\\
&=\left(\sum_{k=0}^\infty \left(2^{ks(x)}P_{2^{k},a}f_{(k)}(x)\right)^{q(x)}\right)^{(1-\lambda)/q(x)}\cdot\\
&\qquad\qquad\cdot \sum_{l=1}^\infty 2^{la(1-\lambda)-ls^-}\left(\sum_{k=0}^\infty(\eta_{k,m}*|2^{(k+l)s(\cdot)}f_{(k+l)}|^\lambda)^{q(x)/\lambda}(x)\right)^{\lambda/q(x)}.
\end{align*}

If $q(x)\le 1$, we obtain in a similar way
{\allowdisplaybreaks\begin{align*}
II&\lesssim \sum_{l=1}^\infty 2^{(la(1-\lambda)-ls^-)q(x)}\\
&\qquad\sum_{k=0}^\infty \left(2^{(k+l)s(x)}P_{2^{k+l},a}f_{(k+l)}(x)\right)^{(1-\lambda)q(x)}
(\eta_{k,m}*|2^{(k+l)s(\cdot)}f_{(k+l)}|^\lambda)^{q(x)}(x)\\
&\le \sum_{l=1}^\infty 2^{(la(1-\lambda)-ls^-)q(x)}
\left(\sum_{k=0}^\infty \left(2^{(k+l)s(x)}P_{2^{k+l},a}f_{(k+l)}(x)\right)^{q(x)}\right)^{1-\lambda}\cdot\\
&\qquad\qquad\cdot\left(\sum_{k=0}^\infty(\eta_{k,m}*|2^{(k+l)s(\cdot)}f_{(k+l)}|^\lambda)^{q(x)/\lambda}(x)\right)^\lambda\\
&=\left(\sum_{k=0}^\infty \left(2^{ks(x)}P_{2^{k},a}f_{(k)}(x)\right)^{q(x)}\right)^{1-\lambda}\cdot\\
&\qquad\qquad\cdot\sum_{l=1}^\infty 2^{(la(1-\lambda)-ls^-)q(x)}
\left(\sum_{k=0}^\infty(\eta_{k,m}*|2^{(k+l)s(\cdot)}f_{(k+l)}|^\lambda)^{q(x)/\lambda}(x)\right)^\lambda
\end{align*}}
and further (with use of Lemma \ref{lem:H})
\begin{align*}
II^{1/q(x)}&\lesssim \left(\sum_{k=0}^\infty \left(2^{ks(x)}P_{2^{k},a}f_{(k)}(x)\right)^{q(x)}\right)^{(1-\lambda)/q(x)}\cdot\\
&\qquad\qquad\cdot\left(\sum_{l=1}^\infty 2^{(la(1-\lambda)-ls^-)q(x)}
\left(\sum_{k=0}^\infty(\eta_{k,m}*|2^{(k+l)s(\cdot)}f_{(k+l)}|^\lambda)^{q(x)/\lambda}(x)\right)^\lambda\right)^{1/q(x)}\\
&\lesssim \left(\sum_{k=0}^\infty \left(2^{ks(x)}P_{2^{k},a}f_{(k)}(x)\right)^{q(x)}\right)^{(1-\lambda)/q(x)}\cdot\\
&\qquad\qquad\cdot\sum_{l=1}^\infty 2^{1/2\cdot(la(1-\lambda)-ls^-)}
\left(\sum_{k=0}^\infty(\eta_{k,m}*|2^{(k+l)s(\cdot)}f_{(k+l)}|^\lambda)^{q(x)/\lambda}(x)\right)^{\lambda/q(x)}.
\end{align*}
If we denote 
$$
F(x):=\left(\sum_{k=0}^\infty \left(2^{ks(x)}P_{2^{k},a}f_{(k)}(x)\right)^{q(x)}\right)^{1/q(x)},\quad x\in\R^n
$$
and
$$
B_{k+l}(x):=|2^{(k+l)s(x)}f_{(k+l)}(x)|,\quad x\in\R^n
$$
we get for $\delta:=-1/2\cdot(a(1-\lambda)-s^-)>0$
\begin{equation}\label{eq:dif10}
II^{1/q(x)}\lesssim F(x)^{1-\lambda}\cdot \sum_{l=1}^\infty 2^{-l\delta}
\left(\sum_{k=0}^\infty(\eta_{k,m}*B_{k+l}^\lambda)^{q(x)/\lambda}(x)\right)^{\lambda/q(x)}.
\end{equation}
We use $\|F_1^{1-\lambda}F_2^\lambda\|_\p\le 2\|F_1\|_\p^{1-\lambda}\|F_2\|_\p^\lambda$, cf. \cite[Lemma 3.2.20]{DHHR},
and suppose that the $L_\p-$ (quasi-)norm is equivalent to an $r$-norm with $0<r\le 1$. Together with Lemma \ref{lem:ConvDHR} we arrive at
{\allowdisplaybreaks\begin{align}
\|II^{1/q(x)}\|^r_\p&\lesssim \|F(x)\|_{\p}^{(1-\lambda)r}\cdot
\left\|\sum_{l=1}^\infty 2^{-l\delta}
\left(\sum_{k=0}^\infty(\eta_{k,m}*B_{k+l}^\lambda)^{q(x)/\lambda}(x)\right)^{1/q(x)}\right\|_\p^{\lambda r}\notag\\
&\lesssim \|F(x)\|_{\p}^{(1-\lambda)r}\cdot
\sum_{l=1}^\infty 2^{-l\delta r}
\left\|\left(\sum_{k=0}^\infty(\eta_{k,m}*B_{k+l}^\lambda)^{q(x)/\lambda}(x)\right)^{1/q(x)}\right\|_\p^{\lambda r}\notag\\
&\lesssim \|f|\Fspqpunkt\|^{(1-\lambda)r} \cdot \sum_{l=1}^\infty 2^{-l\delta r}
\left\| (\eta_{k,m}*B_{k+l}^\lambda(x))_{k=0}^\infty\right\|^r_{L_{p(\cdot)/\lambda}(\ell_{q(\cdot)/\lambda})}\notag\\
&\lesssim \|f|\Fspqpunkt\|^{(1-\lambda)r} \cdot \sum_{l=1}^\infty 2^{-l\delta r}
\left\| (B_{k+l}^\lambda(x))_{k=0}^\infty\right\|^r_{L_{p(\cdot)/\lambda}(\ell_{q(\cdot)/\lambda})}\label{eq:diffetaf}\\
&\lesssim \|f|\Fspqpunkt\|^{(1-\lambda)r} \cdot \sum_{l=1}^\infty 2^{-l\delta r}
\left\| (B_{k}^\lambda(x))_{k=0}^\infty\right\|^r_{L_{p(\cdot)/\lambda}(\ell_{q(\cdot)/\lambda})}\notag\\
&\lesssim \|f|\Fspqpunkt\|^{(1-\lambda)r} \cdot
\left\| B_{k}(x)\right\|_{L_{p(\cdot)}(\ell_{q(\cdot)})}^{\lambda r}\notag\\
&\lesssim \|f|\Fspqpunkt\|^r,\notag
\end{align}}
which finishes the proof.

\emph{Part III.}: We prove the right hand side of \eqref{eq:todo3}.
We follow again essentially \cite[Section 2.5.9]{Triebel1} with some modifications as presented in
\cite{U}. Roughly speaking, compared to the case of constant exponents, only minor modifications
are necessary.

Let $\psi \in C^\infty_0(\R^n)$ with $\psi(x)=1, |x|\le 1$ and $\psi(x)=0, |x|>3/2$. We define
$$
\varphi_0(x)=(-1)^{M+1}\sum_{\mu=0}^{M-1}(-1)^\mu\binom{M}{\mu}\psi((M-\mu)x).
$$
It follows that $\varphi_0\in C^\infty_0(\R^n)$ with $\varphi(x)=0, |x|>3/2$ and $\varphi(x)=1, |x|<1/M$.
We also put $\varphi_j(x)=\varphi_0(2^{-j}x)-\varphi_0(2^{-j+1}x)$ for $j\ge 1.$ This is the decomposition of unity
we used in the definition of $\|f|\Fspqpunkt\|$, cf. Definition \ref{def:BFpunkt}. Recall that due to
\cite{DHR} and \cite{Kempka09}, this (quasi-)norm of $\|f|\Fspqpunkt\|$ does not depend on the choice of the decomposition of unity.

We observe that
$$
\varphi_0(x)=(-1)^{M+1}(\Delta_x^M\psi(0)-(-1)^M),
$$
and
\begin{equation}
\label{eq:diff:rev1}
(\IFT \varphi_j\FT f)(x)=\begin{cases} (\IFT\Delta^M_\xi\psi(0)\FT f)(x)+(-1)^{M+1}f(x),\quad &j=0,\\
(\IFT(\Delta^M_{2^{-j}\xi}\psi(0)-\Delta^M_{2^{-j+1}\xi}\psi(0))\FT f)(x),\quad &j\ge 1.
\end{cases}
\end{equation}
Furthermore, a straightforward calculation shows that
\begin{align}
\notag|(\IFT(\Delta_{2^{-j}\xi}^M\psi(0))\FT f)(x)|&=\left|\sum_{u=0}^M(-1)^u\IFT[\psi((M-u)2^{-j}\cdot)\FT f](x)\right|\\
&\label{eq:diff:rev2}\approx \left|\sum_{u=0}^M(-1)^u\IFT[\psi((M-u)2^{-j}\cdot)]* f(x)\right|\\
\notag&\approx \left|\sum_{u=0}^M(-1)^u\int_{\R^n}\IFT\psi(h) f(x-(M-u)2^{-j}h)dh\right|\\
\notag&= \left|\int_{\R^n} \hat \psi(h) \Delta^M_{2^{-j}h}f(x)dh\right|
\le \int_{\R^n} |\hat \psi(h)|\cdot| \Delta^M_{2^{-j}h}f(x)|dh
\end{align}
holds for every $j\in\N_0$. We denote $g=\hat \psi\in \Sn$ and obtain
\begin{align}
\label{eq:dif7}\|f|\Fspqpunkt\|&\approx\|2^{js(x)}(\IFT \varphi_j \FT f)(x)|L_{p(\cdot)}(\ell_{q(\cdot})\|\\
\notag&\lesssim \norm{f}{L_\p(\R^n)}+\left\|2^{js(x)}\int_{\R^n}|g(h)|\cdot |\Delta^M_{2^{-j}h}
f(x)|dh|L_{p(\cdot)}(\ell_{q(\cdot)})\right\|.
\end{align}
The rest of this part consists essentially of using the property of $g\in \Sn$ to come from
\eqref{eq:dif7} to $\|\cdot\|^{**}$.

We denote 
$$
I_0:=B, \quad I_u:=2^uB\setminus 2^{u-1}B,\quad u\in\N
$$
and use $|g(h)|\le c2^{-ur}, h\in I_u$ with $r$ taken large enough (recall that $g\in \Sn$)
and estimate
\begin{align}
\notag\int_{\R^n}|g(h)|\cdot |\Delta^M_{2^{-j}h}f(x)|dh
&=\sum_{u=0}^\infty \int_{I_u}|g(h)|\cdot |\Delta^M_{2^{-j}h}f(x)|dh\\
\label{eq:diff:rev3}&\lesssim\sum_{u=0}^\infty 2^{-ur}2^{jn} \int_{2^{u-j}B}|\Delta^M_{h}f(x)|dh\\
\notag&=\sum_{u=0}^\infty 2^{u(n-r)}2^{-(u-j)n} \int_{2^{u-j}B}|\Delta^M_{h}f(x)|dh.
\end{align}

We put
$$
 G_j(x):=2^{js(x)}|(\IFT(\Delta_{2^{-j}\xi}^M\psi(0))\FT f)(x)|,\quad j\in\N_0
$$
and
$$
g_k(x):=2^{ks(x)}2^{kn}\int_{2^{-k}B} |\Delta^M_{h}f(x)|dh,\quad k\in\Z.
$$
Using \eqref{eq:diff:rev1}, \eqref{eq:diff:rev2} and \eqref{eq:diff:rev3}, we obtain the estimate
\begin{align}
\notag G_j(x)&\lesssim 2^{js(x)}\sum_{u=0}^\infty 2^{u(n-r)}2^{-(u-j)n} \int_{2^{u-j}B}|\Delta^M_{h}f(x)|dh\\
\label{eq:G1}&=\sum_{k=-\infty}^j 2^{(j-k)s(x)} 2^{(j-k)(n-r)}2^{ks(x)}2^{kn}\int_{2^{-k}B} |\Delta^M_{h}f(x)|dh\\
\notag &=\sum_{k=-\infty}^j 2^{(j-k)(s(x)+n-r)}g_k(x)\le \sum_{k=-\infty}^\infty 2^{|j-k|\cdot(s(x)+n-r)}g_k(x).
\end{align}
Choosing $r>s^++n$ and applying Lemma \ref{_LM_Lemma2} then finishes the proof.

\emph{Part IV.} The proof of the left hand side of \eqref{eq:todo4} 
follows in the same manner as in Part II. We shall describe the necessary
modifications. First, let us mention, that the condition $q^+<\infty$ was used only in the application of Lemma \ref{lem:ConvDHR}.
In the rest of the arguments also the case $q(x)=\infty$ may be incorporated with only slight change of notation.

Let us put
$$
f^{(k)}(x):=2^{ks(x)}\int_B |\Delta^M_{2^{-k}h}f(x)|dh, \quad x\in\R^n.
$$
We obtain (in analogue to \eqref{eq:proof1})
\begin{align*}
f^{(k)}&\le f^{(k),I}+f^{(k),II}:=\sum_{l=-\infty}^0 2^{ks(x)}\int_B |\Delta^M_{2^{-k}h}f_{(k+l)}(x)|dh\\
&\qquad+\sum_{l=1}^\infty 2^{ks(x)}\int_B |\Delta^M_{2^{-k}h}f_{(k+l)}(x)|dh.
\end{align*}
We estimate the first sum using \eqref{eq:dif1} and get
$$
f^{(k),I}\lesssim \sum_{l=-\infty}^0 2^{l(M-s^+)}g^1_{k+l}=\sum_{u=0}^k 2^{(u-k)(M-s^+)}g^1_u\le \sum_{u=0}^\infty 2^{-|u-k|(M-s^+)}g^1_u,
$$
where $g^1_{u}:=2^{us(x)}P_{2^{u},a}f_{(u)}(x)$. The application of Lemma \ref{_LM_Lemma2} 
and Theorem \ref{thm:sx:loc1} with $a>0$ large enough gives
\begin{align*}
\|\left(f^{(k),I}\right)_{k=0}^\infty|\ellqp\|\lesssim \|\left(g_u\right)_{u=0}^{\infty}|\ellqp\|\lesssim \|f|\Bspqpunkt\|.
\end{align*}

To estimate $f^{(k),II}$, we proceed as in the Step 4 of Part II.
If $p^->1$, we choose again $\lambda=1$, otherwise we take $0<\lambda<p^-$ and 
$$
a>\frac{n+c_{\rm log}(1/q)}{p^-}+c_{\rm log}(s)
$$
such that $a(1-\lambda)<s^-$. This is possible due to \eqref{eq:thm:cond2}.

We use \eqref{eq:dif2} with $m>\max(n+c_{\log}(1/q),c_{\log}(s))$ to get
\begin{align}\label{eq:dif11}
\begin{split}
	f^{(k),II}&\lesssim \sum_{l=1}^\infty 2^{la(1-\lambda)-ls^-}\left(2^{(k+l)s(x)}P_{2^{k+l},a}f_{(k+l)}(x)\right)^{1-\lambda}\cdot\\
&\qquad\qquad\qquad\qquad\quad\cdot(\eta_{k,m}*|2^{(k+l)s(\cdot)}f_{(k+l)}(\cdot)|^\lambda)(x)\\
\end{split}\\
\notag&=\sum_{l=1}^\infty 2^{la(1-\lambda)-ls^-}(g^1_{k+l}(x))^{1-\lambda}\cdot(\eta_{k,m}*(g^2_{k+l})^\lambda)(x),
\end{align}
where $g^2_{k+l}(x):=|2^{(k+l)s(x)}f_{(k+l)}(x)|$. We take the $\ellqp$ (quasi-)norm of the last expression - and assume that
it is equivalent to some $r$-norm. This gives for $\delta:=s^--a(1-\lambda)>0$ the following estimate
{\allowdisplaybreaks\begin{align}
\|f^{(k),II}&|\ellqp\|^r\lesssim \sum_{l=1}^\infty 2^{-l\delta r}\|(g^1_{k+l}(x))^{1-\lambda}\cdot(\eta_{k,m}*(g^2_{k+l})^\lambda)(x)|\ellqp\|^r\notag\\
&\lesssim \sum_{l=1}^\infty 2^{-l\delta r}\|g^1_{k+l}|\ellqp\|^{(1-\lambda)r}\cdot \|[\eta_{k,m}*(g^2_{k+l})^\lambda]^{1/\lambda}|\ellqp\|^{\lambda r}\notag\\
&\lesssim \sum_{l=1}^\infty 2^{-l\delta r}\|g^1_{k}|\ellqp\|^{(1-\lambda)r} \cdot \|\eta_{k,m}*(g^2_{k+l})^\lambda|\ell_{\q/\lambda}(L_{\p/\lambda})\|^{ r}\notag\\
&\lesssim \sum_{l=1}^\infty 2^{-l\delta r}\|g^1_{k}|\ellqp\|^{(1-\lambda)r} \cdot \|(g^2_{k+l})^\lambda|\ell_{\q/\lambda}(L_{\p/\lambda})\|^{ r}\label{eq:diffetab}\\
&\lesssim \sum_{l=1}^\infty 2^{-l\delta r}\|g^1_{k}|\ellqp\|^{(1-\lambda)r} \cdot \|g^2_{k+l}|\ellqp\|^{\lambda r}\notag\\
&\lesssim \|f|\Bspqpunkt\|^r.\notag
\end{align}}
We have used Lemma \ref{Lemma_4} and Lemma \ref{lem:Hoelder}.

\emph{Part V.} The right hand side inequality of \eqref{eq:todo4} follows also along the same line as in Part III.
We just combine \eqref{eq:G1} with the choice $r>s^++n$ and apply Lemma \ref{_LM_Lemma2}.
\end{proof}

\subsection{Proof of Theorem \ref{thm:dif}}\label{Sec:Diff:q}

This section is devoted to the proof of Theorem \ref{thm:dif}. We start with the case of constant $q$.
In that case, the usual Hardy-Littlewood maximal operator
$$
Mf(x)=\sup_{r>0}\frac{1}{|B(x,r)|}\int_{B(x,r)}|f(y)|dy
$$
is bounded on $\ell_q(L_{p(\cdot)})$ and $L_{p(\cdot)}(\ell_q)$. Indeed, the following lemma is a consequence of \cite{CUFMP}
and \cite[Theorem 4.3.8]{DHHR}.
\begin{lem}\label{Lem:Max}
(i) Let $p\in \Plog$ with $1<p^-\le p^+<\infty$ and $1<q<\infty$.
Then
$$
\norm{(M f_j)_{j=-\infty}^\infty}{L_{p(\cdot)}(\ell_q)}\lesssim \norm{(f_j)_{j=-\infty}^\infty}{L_{p(\cdot)}(\ell_q)}
$$
for all $(f_j)_{j=-\infty}^\infty\in L_{p(\cdot)}(\ell_q)$.

(ii) Let $p\in \Plog$ with $p^->1$ and $0<q\le \infty$. Then
$$
\norm{(M f_j)_{j=-\infty}^\infty}{\ell_q(L_{p(\cdot)})}\lesssim \norm{(f_j)_{j=-\infty}^\infty}{\ell_q(L_{p(\cdot)})}
$$
for all $(f_j)_{j=-\infty}^\infty\in \ell_q(L_{p(\cdot)})$.
\end{lem}
\emph{Proof of Theorem \ref{thm:dif}. }
With the help of Lemma \ref{Lem:Max}, we prove Theorem \ref{thm:dif} for $q$ constant. In view of Lemma \ref{Lem:First},
it is enough to prove 
\begin{equation}\label{eq:todo_F1}
\|f|\Fspqpunkt\|^{**}\lesssim \|f|\Fspqpunkt\|
\end{equation}
and a corresponding analogue for the $B$-spaces.

\emph{Part I. }In this part we point out the necessary modifications in the proof of Lemma \ref{Lem:First} to obtain a characterization by ball means of differences for $B^ \s_{\p,q}(\R^n)$ and $F^\s_{\p,q}(\R^n)$. The proof follows the scheme of Part II of the proof of Lemma \ref{Lem:First}. 
We start with $\displaystyle \sum_{k=-\infty}^\infty$ instead of $\displaystyle\sum_{k=0}^\infty.$ With this modification
the Steps 1-3 go through without any other changes and we obtain \eqref{eq:fin1} again (just recall that $f_{(k)}=0$
if $k<0$).

Due to the boundedness of the maximal operator there is no need for the use of $r$-trick and convolution with $\eta_{\nu,m}$.
The analogue of \eqref{eq:dif6}, \eqref{eq:dif2} and \eqref{eq:dif8} now reads as follows:
\begin{align*}
2^{ks(x)}&\int_B|\Delta^M_{2^{-k}h}f_{(k+l)}(x)|dh\\
&\hspace{-1em}\lesssim\left(2^{ks(x)}2^{la}P_{2^{k+l},a}f_{(k+l)}(x)\right)^{1-\lambda}
\sum_{j=0}^M c_{j,M}\int_B2^{ks(x+j2^{-kh})\lambda}|f_{(k+l)}(x+j2^{-k}h)|^\lambda dh\\
&\hspace{-1em}\le 2^{la(1-\lambda)-ls^-}\left(2^{(k+l)s(x)}P_{2^{k+l},a}f_{(k+l)}(x)\right)^{1-\lambda}\cdot\\
&\hspace{6em}\cdot\sum_{j=0}^M c_{j,M}\int_B2^{(k+l)s(x+j2^{-kh})\lambda}|f_{(k+l)}(x+j2^{-k}h)|^\lambda dh\\
&\hspace{-1em}\lesssim 2^{la(1-\lambda)-ls^-}\left(2^{(k+l)s(x)}P_{2^{k+l},a}f_{(k+l)}(x)\right)^{1-\lambda}
\sum_{j=0}^M c_{j,M} M(|2^{(k+l)s(\cdot)}f_{(k+l)}(\cdot)|^\lambda)(x),
\end{align*}
where we used H\"older's regularity of $\s$, see \eqref{sxClogeigenschaft}. As a consequence, we obtain
\begin{equation*}
II^{1/q(x)}\lesssim F(x)^{1-\lambda}\cdot \sum_{l=1}^\infty 2^{-l\delta}
\left(\sum_{k=-l}^\infty(M B_{k+l}^\lambda)^{q(x)/\lambda}(x)\right)^{\lambda/q(x)}.
\end{equation*}
instead of \eqref{eq:dif10}. The rest then follows in the same manner with the help of Lemma \ref{Lem:Max}
and the proof of \eqref{eq:todo_F1} is finished.

The proof of
\begin{equation*}
\|f|\Bspqpunkt\|^{**}\lesssim \|f|\Bspqpunkt\|
\end{equation*}
follows along the same lines. Especially, we get
\begin{align*}
f^{(k),II}&\lesssim \sum_{l=1}^\infty 2^{la(1-\lambda)-ls^-}(g^1_{k+l}(x))^{1-\lambda}\cdot(M(g^2_{k+l})^\lambda)(x)
\end{align*}
instead of \eqref{eq:dif11}. The rest follows again by Lemma \ref{Lem:Max}.

\emph{Part II. }Finally, we present how the characterization for $q$ constant can help us to improve on the case of variable
exponent $\q$.

In view of Lemma \ref{Lem:First}, it is enough to show that

\begin{align*}
\left\|\left(\sum_{k=-\infty}^0 2^{ks(x)q(x)}\left(d^M_{2^{-k}}f(x)\right)^{q(x)}\right)^{1/q(x)}
\biggl|L_{p(\cdot)}(\R^n)\right\|\lesssim \|f|\Fspqpunkt\|.
\end{align*}
But this is a consequence of
\begin{align*}
&\left\|\left(\sum_{k=-\infty}^0 2^{ks(x)q(x)}\left(d^M_{2^{-k}}f(x)\right)^{q(x)}\right)^{1/q(x)}
\biggl|L_{p(\cdot)}(\R^n)\right\|\\
&\qquad\qquad\lesssim\left\|\left(\sum_{k=-\infty}^0 2^{k(s(x)-\varepsilon)q^-}\left(d^M_{2^{-k}}f(x)\right)^{q^-}\right)^{1/q^-}
\biggl|L_{p(\cdot)}(\R^n)\right\|\\
&\qquad\qquad \lesssim\|f|F^{s(\cdot)-\varepsilon}_{p(\cdot),q^-}(\R^n)\|
\lesssim\|f|F^{s(\cdot)}_{p(\cdot),q(\cdot)}(\R^n)\|,
\end{align*}
where $\varepsilon>0$ is small enough and we used the differences characterization for fixed $q$ and a trivial embedding theorem.

The same arguments apply for the Besov spaces and the proof is finished.\qed
\begin{rem}
The somewhat complicated proof of Theorem \ref{thm:dif} would work more direct and simpler if we could use versions of Lemmas \ref{Lemma_4} and \ref{lem:ConvDHR} in \eqref{eq:diffetaf} and \eqref{eq:diffetab} where the $\ell_\q$ summation runs over $\nu\in\Z$.\\
For Triebel-Lizorkin spaces there seems to exist such an extension \cite{Lars_p}, but for Besov spaces the proof of Lemma \ref{Lemma_4} in \cite{AlmeidaHasto} seems to be to customized to the situation $\nu\in\N_0$.  
\end{rem}
\subsection{Ball means of differences for 2-microlocal spaces}\label{sec:2mldiff}
As already remarked in Section \ref{sec:not2ml} all the proofs for spaces of variable smoothness do also serve for 2-microlocal spaces. One just has to use the definition of admissible weight sequences and the property \eqref{sxClogeigenschaft}, see Remark \ref{rem:2mlsx}.\\
First of all we give the notation for the (quasi-)norms. For simplicity we just use the discrete versions, although it is also possible to give continuous versions of 2-microlocal weights, see \cite[Definition 4.1]{UlRau}. In analogy to the spaces of variable smoothness we introduce the following norms
\begin{align*}
\norm{f}{\Bwpqpunkt}^{**}&=\norm{f}{L_\p(\R^n)}+\norm{\left(w_k(x)d_{2^{-k}}^Mf(x)\right)_{k=-\infty}^\infty}{\ellqp}\\
\intertext{and}
\norm{f}{\Fwpqpunkt}^{**}&=\norm{f}{L_\p(\R^n)}+\norm{\left(w_k(x)d_{2^{-k}}^Mf(x)\right)_{k=-\infty}^\infty}{\ellpq}.
\end{align*}
Finally, the preceding calculations show that the following theorem is true.
\begin{thm}
(i) Let $p,q\in\Plog$ with $p^+,q^+<\infty$ and $\vek{w}\in\mgk$. Let $M>\alpha_2$ and 
\begin{equation}\label{eq:thm:cond12ml}
\alpha_1>\sigma_{p^-,q^-}\cdot\left[1+\frac{\alpha}{n}\cdot \min(p^-,q^-)\right].
\end{equation}
Then
$$
\Fwpqpunkt=\{f\in L_{\p}(\R^n):\|f|\Fwpqpunkt\|^{**}<\infty\}
$$
and $\|\cdot|\Fwpqpunkt\|$ and $\|\cdot|\Fwpqpunkt\|^{**}$ are equivalent on $\Fwpqpunkt$.

(ii) Let $p,q\in\Plog$ and $\vek{w}\in\mgk$. Let $M>\alpha_2$ and
\begin{equation}\label{eq:thm:cond22ml}
\alpha_1>\sigma_{p^-}\cdot \left[1+\frac{c_{\rm log}(1/q)}{n}+\frac{\alpha}{n}\cdot p^-\right].
\end{equation}
Then
$$
\Bwpqpunkt=\{f\in L_{\p}(\R^n):\|f|\Bwpqpunkt\|^{**}<\infty\}
$$
and $\|\cdot|\Bwpqpunkt\|$ and $\|\cdot|\Bwpqpunkt\|^{**}$ are equivalent on $\Bwpqpunkt$.
\end{thm}
\begin{rem}
Again, if $\min(p^-,q^-)\geq1$ in the F-case, or $p^-\geq1$ in the B-case, then the conditions \eqref{eq:thm:cond12ml} and \eqref{eq:thm:cond22ml} simplify to $\alpha_1>0$. In the case of constant exponents $p,q$ we obtain similar results to \cite{Bes1} and \cite{VehelSeuret03}. 
\end{rem}

\subsection{Lemmas}\label{Sec:Lemmas}
The following lemma is a variant of Lemma 6.1 from \cite{DHR}.
\begin{lem}\label{Lem:eta1}
Let $s\in C^{\log}_{loc}(\R^n)$ and let $R\geq c_{\log}(s)$ , where $c_{\log}(s)$ is the constant from \eqref{eq:defclog} for 
$\s$. Then
$$
2^{\nu s(x)}\eta_{\nu,m+R}(x-y)\le c\,2^{\nu s(y)}\eta_{\nu,m}(x-y)
$$ 
holds for all $x,y\in\R^n$ and $m\in\N_0$.
\end{lem}

\begin{lem}\label{lem:r-trick}
Let $r>0$, $\nu\ge 0$ and $m>n$. Then there exists $c>0$, which depends only on $m,n$ and $r$, such that for all $g\in S'(\R^n)$
with $\supp \hat g\subset \{\xi\in\R^n:|\xi|\le 2^{\nu+1}\}$, we have
$$
 |g(x)|\le c(\eta_{\nu,m}*|g|^r(x))^{1/r},\quad x\in\R^n.
$$
\end{lem}
The following lemma is the counterpart to Lemma \ref{Lemma_4} for Triebel-Lizorkin spaces.
\begin{lem}[\cite{DHR}, Theorem 3.2]\label{lem:ConvDHR}
	Let $p,q\in\Plog$ with $1<p^-\leq p^+<\infty$ and $1<q^-\leq q^+<\infty$. Then the inequality
	\begin{align*}
		\norm{(\eta_{\nu,m}\ast f)_{\nu=0}^\infty}{\ellpq}\leq c\norm{(f_\nu)_{\nu=0}^\infty}{\ellpq}
	\end{align*}
	holds for every sequence $(f_\nu)_{\nu\in\N_0}$ of $L_1^{loc}(\R^n)$ functions and $m>n$.
\end{lem}

The following lemma is well known (cf. \cite{Triebel1}). We sketch its proof for the sake of completeness.

\begin{lem}\label{lem:dif_Peetre}
Let $a,b>0$, $M\in \N$ and $h\in\R^n$. Let $f\in S'(\R^n)$ with $\supp \hat f\subset\{\xi\in\R^n:|\xi|\le b\}$.
Then there is a constant $C>0$ independent of $f, b$ and $h$, such that
$$
|\Delta^M_h f(x)|\le C\max(1,|bh|^a)\cdot\min(1,|bh|^M)P_{b,a}f(x)
$$
holds for every $x\in\R^n$.
\end{lem}
\begin{proof}
The estimate
\begin{align*}
|f(x+jh)|&=\frac{|f(x+jh)|}{1+|jbh|^a}\cdot(1+|jbh|^a)\le (1+|Mbh|^a)\sup_{z\in\R^n}\frac{f(x-z)}{1+|bz|^a}\\
&\lesssim \max(1,|bh|^a)P_{b,a}f(x),\quad j=0,\dots,M,
\end{align*}
holds for all the admissible parameters even without the assumption on $\hat f$.

Hence we need to prove only
\begin{equation}\label{eq:proof:D2}
|\Delta^M_h f(x)|\le C\max(1,|bh|^a)\cdot|bh|^M\cdot P_{b,a}f(x).
\end{equation}
Using the Taylor formula for the (analytic) function $f$, we obtain by direct calculation
\begin{align*}
|\Delta^M_h f(x)|&\le c|h|^M \sup_{|\alpha|=M}\sup_{|y|\le M|h|}\frac{|(D^\alpha f)(x-y)|}{1+|by|^a}\cdot(1+|by|^a)\\
&\le c' |h|^M\max(1,|bh|^a)\cdot\sup_{|\alpha|=M}\sup_{|y|\le M|h|}\frac{|(D^\alpha f)(x-y)|}{1+|by|^a}.
\end{align*}
If $\supp \hat g\subset\{\xi\in\R^n:|\xi|\le 1\}$, then this may be combined with the Nikol'skij inequality,
cf. \cite[Section 1.3.1]{Triebel1}, in the form
$$
\sup_{|\alpha|=M}\sup_{z\in\R^n}\frac{|(D^\alpha g)(x-z)|}{1+|z|^a}\lesssim \sup_{z\in\R^n}\frac{|g(x-z)|}{1+|z|^a}
$$
to obtain
\begin{equation}\label{eq:proof:D1}
|\Delta^M_h g(x)|\le c'' |h|^M \max(1,|h|^a)\cdot\sup_{z\in\R^n}\frac{|g(x-z)|}{1+|z|^a}.
\end{equation}
If $\supp \hat f\subset\{\xi\in\R^n:|\xi|\le b\}$, we define $g(x)=f(x/b)$, apply \eqref{eq:proof:D1}
together with $\Delta_h^Mf(x)=\Delta_{bh}^Mg(bx)$ and obtain
$$
|\Delta^M_h f(x)|\lesssim |bh|^M\max(1,|bh|^a)\sup_{z\in\R^n}\frac{|g(bx-z)|}{1+|z|^a}.
$$
From this \eqref{eq:proof:D2} follows and the proof is then complete.
\end{proof}

The following lemma resembles Lemma A.3 of \cite{DHR}. 

\begin{lem}\label{lem:conv}
Let $k\in\Z$, $l\in \N_0$ and $m>n$. Then
$$
\eta_{k+l,m}*[2^{kn}\chi_{2^{-k}B}]\lesssim \eta_{k,m}.
$$
\end{lem}
\begin{proof}
Using dilations, we may suppose that $k=0$. If $|x|\le 2$, then
$$
\int_{\{y:|x-y|\le 1\}}2^{nl}(1+2^l|y|)^{-m}dy \le 
\int_{y\in\R^n}2^{nl}(1+2^l|y|)^{-m}dy \lesssim (1+|x|)^{-m}.
$$
If $|x|>2$ and $|x-y|\le 1$, we obtain $1+2^l|y|\gtrsim 1+2^l|x|$ and $2^{nl}(1+2^l|x|)^{-m}\lesssim (1+|x|)^{-m}$.
This immediately implies that
\begin{align*}
\int_{\{y:|x-y|\le 1\}}2^{nl}(1+2^l|y|)^{-m}dy\lesssim
\int_{\{y:|x-y|\le 1\}}(1+|x|)^{-m}dy\lesssim (1+|x|)^{-m}.
\end{align*}
\end{proof}
\begin{rem} Another way, how to prove Lemma \ref{lem:conv} is to use the inequality
$\chi_B(x)\le 2^{m}\eta_{0,m}(x)$ and apply Lemma A.3 of \cite{DHR}.
\end{rem}
The following Lemma is quite simple and we leave out its proof.
\begin{lem}\label{lem:H}
Let $0<q<\infty$, $\delta>0$ and let $(a_l)_{l\in\N}$ be a sequence of non-negative real numbers. Then
$$
\left(\sum_{l=1}^\infty 2^{-l\delta q}a_l\right)^{1/q}\lesssim \sum_{l=1}^\infty 2^{-l\delta/2}a_l^{1/q},
$$
where the constant involved depends only on $\delta$ and $q$.
\end{lem}
Finally, we shall need a certain version of H\"older's inequality for $\ellqp$ spaces.
\begin{lem}\label{lem:Hoelder}
Let $p,q\in\P$ and let $0<\lambda<1$. Then
\begin{equation}\label{eq:Hoelder}
\|f_k\cdot g_k|\ellqp\|\le 2^{1/q^-} \|f_k^{1/(1-\lambda)}|\ellqp\|^{1-\lambda}\cdot \|g_k^{1/\lambda}|\ellqp\|^\lambda
\end{equation}
holds for all sequences of non-negative functions $(f_k)_{k\in\N_0}$ and $(g_k)_{k\in\N_0}$.
\end{lem}
\begin{proof}
Due to the homogeneity, we may assume that
$$
\norm{f_k^{1/(1-\lambda)}}{\ellqp}=\norm{g_k^{1/\lambda}}{\ellqp}=1.
$$
Then for every $\varepsilon>0$, there exist two sequences of positive real numbers 
$(\lambda_k)_{k\in\N_0}$ and $(\mu_k)_{k\in\N_0}$, such that
$$
\sum_{k=0}^\infty \lambda_k < 1+\varepsilon, \quad \sum_{k=0}^\infty \mu_k < 1+\varepsilon
$$
and
$$
\varrho_{\p}\left(\frac{f_k^{1/(1-\lambda)}}{\lambda_k^{1/q(\cdot)}}\right)\le 1,\quad
\varrho_{\p}\left(\frac{g_k^{1/\lambda}}{\mu_k^{1/q(\cdot)}}\right)\le 1.
$$
We put 
$$
c:=2^{1/q^-}\quad\text{and}\quad\gamma_k:=\frac{\lambda_k+\mu_k}{2}\ge \frac{\lambda_k+\mu_k}{c^{q(x)}}
$$
and use the Young inequality in the form
$$
[f_k(x)g_k(x)]^{p(x)}\le (1-\lambda)f_k(x)^{p(x)/(1-\lambda)}+\lambda g_k(x)^{p(x)/\lambda}
$$
to obtain 
\begin{align*}
\int_{\R^n}\left(\frac{f_k(x)g_k(x)}{c\gamma_k^{1/{q(\cdot)}}}\right)^{p(x)}dx
&\le (1-\lambda)\int_{\R^n}\frac{f_k(x)^{p(x)/(1-\lambda)}}{c^{p(x)}\gamma_k^{p(x)/q(x)}}dx
+ \lambda\int_{\R^n}\frac{g_k(x)^{p(x)/\lambda}}{c^{p(x)}\gamma_k^{p(x)/q(x)}}dx\\
&\le (1-\lambda)\int_{\R^n}\frac{f_k(x)^{p(x)/(1-\lambda)}}{\lambda_k^{p(x)/q(x)}}dx
+ \lambda\int_{\R^n}\frac{g_k(x)^{p(x)/\lambda}}{\mu_k^{p(x)/q(x)}}dx\le 1.
\end{align*}
Furthermore, the estimate
$$
\sum_{k=0}^\infty \gamma_k<\frac{2(1+\varepsilon)}{2}=1+\varepsilon
$$
finishes the proof of \eqref{eq:Hoelder} with the constant $c=2^{1/q^-}$.
\end{proof}
\textbf{Acknowledgement: }The first author acknowledges the financial support provided
by the DFG project HA 2794/5-1 "Wavelets and function spaces on domains". Furthermore, the first author thanks the RICAM for his hospitality and support during a short term visit in Linz.\\
The second author acknowledges the financial support provided by the FWF project
Y 432-N15 START-Preis "Sparse Approximation and Optimization in High Dimensions".\\
We thank the anonymous referee for pointing the reference \cite{Drihem} out to us.

\end{document}